\documentclass{amsart}

\pdfoutput=1

\usepackage{amsmath, mathtools, comment}
\usepackage{amsfonts}
\usepackage[numbers]{natbib}
\usepackage{amssymb}
\usepackage{amsthm}
\usepackage{tikz-cd}
\usetikzlibrary{decorations.pathmorphing}
\usepackage{tikz}
\usepackage{amsxtra}
\usepackage{wrapfig}
\usepackage{color}
\usepackage{url}
\usepackage[hidelinks]{hyperref}
\usepackage[percent]{overpic}
\usepackage{mathrsfs}
\usepackage{stmaryrd, accents}

\usepackage{setspace}

\newtheorem{thm}{Theorem}[section]
\newtheorem{prop}[thm]{Proposition}

\newtheorem{cor}[thm]{Corollary}

\newtheorem{lem}[thm]{Lemma}
\newtheorem{thmx}{Theorem}

\theoremstyle{definition}
\newtheorem{defn}[thm]{Definition}

\theoremstyle{remark}
\newtheorem{ex}[thm]{Example}
\newtheorem{rmk}[thm]{Remark}

\newcommand{\cat}[1]{{\mathbf{#1}}}
\newcommand{\p}{}

\newcommand{\into}{\hookrightarrow}

\newcommand{\lot}{\otimes^{\mathbb{L}}}

\newcommand{\per}{{\ensuremath{\cat{Perf}}}\kern 1pt}

\newcommand{\Mod}{\cat{Mod}\text{-}}
\newcommand{\Comod}{\cat{Comod}\text{-}}
\newcommand{\uhom}{\undertilde{\hom}}
\newcommand{\ootimes}{\widetilde{\otimes}}

\newcommand{\Ctrmod}{\cat{Ctrmod}\text{-}}
\newcommand{\comod}{\cat{Comod}\text{-}}

\newcommand{\pcMod}{\cat{pcMod}\text{-}}

\DeclareMathOperator{\id}{id}

\let\hom\relax\newcommand{\hom}{\mathrm{Hom}}
\newcommand{\cohom}{\mathrm{Cohom}}
\newcommand{\enn}{\mathrm{End}}

\DeclareMathOperator{\ext}{Ext}

\DeclareMathOperator{\map}{Hom}
\DeclareMathOperator{\gldim}{gl.dim.}
\DeclareMathOperator{\HH}{HH}
\DeclareMathOperator{\hh}{hh}
\DeclareMathOperator{\hg}{h}

\newcommand{\C}{\mathcal{C}}

\newcommand{\Q}{\mathbb{Q}}
\newcommand{\Z}{\mathbb{Z}}

\renewcommand{\P}{\mathbb{P}}
\newcommand{\R}{{\mathrm{\normalfont\mathbb{R}}}}

\newcommand{\rmap}{\R\kern -1.5pt \map}

\newcommand{\pvd}{{\ensuremath{\cat{pvd}}}\kern 1pt}
\newcommand{\thick}{{\ensuremath{\cat{thick}}}}
\newcommand{\dco}{\ensuremath{D^\mathrm{co}}}
\newcommand{\dct}{\ensuremath{D^\mathrm{ctr}}}
\newcommand{\dpc}{\ensuremath{D^\mathrm{pc}}}
\newcommand{\hqfco}{{\ensuremath{\cat{hqf}^\mathrm{co}}}}
\newcommand{\hqfct}{{\ensuremath{\cat{hqf}^\mathrm{ctr}}}}
\newcommand{\fd}{{\ensuremath{\cat{fd}}}}
\newcommand{\rf}{{\ensuremath{\cat{ref}}}\kern 1pt}

\newcommand{\bimod}{\text{-}\cat{Mod}\text{-}}
\newcommand{\bicomod}{\text{-}\cat{Comod}\text{-}}
\newcommand{\bictmod}{\text{-}\cat{Ctrmod}\text{-}}

\makeatletter
\newcommand{\holim@}[2]{%
	\vtop{\m@th\ialign{##\cr
			\hfil$#1\operator@font holim$\hfil\cr
			\noalign{\nointerlineskip\kern1.5\ex@}#2\cr
			\noalign{\nointerlineskip\kern-\ex@}\cr}}%
}
\newcommand{\holim}{%
	\mathop{\mathpalette\holim@{\leftarrowfill@\textstyle}}\nmlimits@
}
\newcommand{\hocolim}{%
	\mathop{\mathpalette\holim@{\rightarrowfill@\textstyle}}\nmlimits@
}
\makeatother

\numberwithin{equation}{section}

\usepackage[utf8]{inputenc}
\usepackage[T1]{fontenc}
\usepackage{textcomp}

\usepackage[nolabel]{showlabels}

\title[Nonsmooth Calabi--Yau structures]{Nonsmooth Calabi--Yau structures for algebras and coalgebras}
\author[Booth, Chuang, Lazarev]{M. Booth, J. Chuang, A. Lazarev}

\address{\mbox{Department of Mathematics,
Huxley Building,
Imperial College London,
SW7 2AZ}
\indent Heilbronn Institute for Mathematical Research,
Bristol, BS8 1UG}
\email{matt.booth@imperial.ac.uk}
\address{City St George's, University of London,
Northampton Square, EC1V 0HB}
\email{joseph.chuang.1@citystgeorges.ac.uk}
\address{School of Mathematical Sciences,
Lancaster University,
Lancaster,
LA1 4YF}
\email{a.lazarev@lancaster.ac.uk}

\subjclass{18N40, 16E65, 14A22, 57P10}

\keywords{Calabi--Yau algebras, Koszul duality, Poincar\'e duality spaces}

\thanks{This work was supported by EPSRC grants EP/N016505/1 and EP/T029455/1. This work was supported by the Additional Funding Programme for Mathematical Sciences, delivered by EPSRC (EP/V521917/1) and the Heilbronn Institute for Mathematical Research.}

\begin{document}

\begin{abstract}
{We show that generalised Calabi--Yau dg (co)algebras are Koszul dual to generalised symmetric dg (co)algebras, without needing to assume any smoothness or properness hypotheses. Similarly, we show that Gorenstein and Frobenius are Koszul dual properties. As an application, we give a new characterisation of Poincaré duality spaces, which extends a theorem of Félix--Halperin--Thomas to the non-simply connected setting.}
\end{abstract}

	\maketitle

\tableofcontents

\section{Introduction}

Over the past twenty years, many authors have studied various different kinds of noncommutative Calabi--Yau structures on dg categories \cite{ginzburg, KS, kellerCYC, vdb, BD, BDII, KTV, KTVII}. Loosely, a Calabi--Yau structure should be thought of as a derived noncommutative version of a trivialisation of the canonical sheaf; for homologically smooth or proper dg categories one often asks for finer data, usually in the form of a lift of a Hochschild class to (negative) cyclic homology.

\p In this paper, we define and study a very general kind of Calabi--Yau dg (co)algebra; our definition is a generalisation to the non-smooth setting of the `left Calabi--Yau' structures of \cite{BD}. We also study a version of the `right Calabi--Yau' condition for dg (co)algebras - which we call \textbf{symmetric}, since it is a natural generalisation of the corresponding discrete notion for finite dimensional algebras (also known as symmetric Frobenius) - and show that these two concepts correspond across Koszul duality.

\p Koszul duality has many different meanings, but for us in this paper it refers to the equivalence of $\infty$-categories or of model categories
$$\mathbf{dgAlg}^\mathrm{aug}_k \longleftrightarrow \mathbf{dgCog}^\mathrm{conil}_k$$between augmented dg-$k$-algebras and conilpotent dg-$k$-coalgebras, for a fixed field $k$. The functors in question are given by the bar and cobar constructions. This equivalence in particular underlies the modern approach to noncommutative derived deformation theory \cite[\S3]{DAGX}. DG coalgebras can also be thought of as pseudocompact dg algebras (or more accurately, their continuous linear duals) and we switch between these perspectives when convenient.

\p Specifically, we prove the following theorem (part of Theorem \ref{thma} below). Let $A$ be an augmented dg algebra and $C\coloneqq BA$ its Koszul dual conilpotent dg coalgebra. Then:
\begin{enumerate}
    \item $A$ is nonsmooth Calabi--Yau if and only if $C$ is symmetric.
    \item $C$ is nonsmooth Calabi--Yau if and only if $A$ is symmetric.
\end{enumerate}
We remark that one can use this theorem to give a Morita duality result, stating that if $A$ is a smooth local Calabi--Yau dg algebra then $\pvd(A)$ is a right Calabi--Yau dg category (see \ref{moritaduality} and \ref{BDmorita} for the details).

\p It will be key for us to work with nonsmooth dg algebras (and dually, non-proper dg coalgebras). For example, the dual numbers $k[\varepsilon]/\varepsilon^2$ is a nonsmooth Calabi--Yau dg algebra, which leads to the surprising conclusion that its Koszul dual, the power series ring $k\llbracket x\rrbracket$, is a pseudocompact symmetric algebra. One can also prove this directly from the definitions (\ref{funnydual}); there is no contradiction here as the `pseudocompact linear dual' is not computed as the linear dual of the underlying vector space. More generally, we prove that the pseudocompact graded algebra of power series $k\llbracket x_1,\ldots, x_n\rrbracket$, with $x_i$ in degree $d_i$, is $(n-\sum_id_i)$-symmetric (\ref{pwfrobcor}). In a loose sense this can be thought of as a derived local complete intersection property for the associated formal derived scheme. In fact, a similar argument shows that a discrete commutative complete local Gorenstein $k$-algebra also has this self-duality property (\ref{gorrmk}), which is in some sense a manifestation of Matlis duality for zero-dimensional Gorenstein rings.


\p We also develop one-sided versions of our generalised Calabi--Yau conditions; the one-sided version of nonsmooth Calabi--Yau is \textbf{Gorenstein} and the one-sided version of symmetric is simply \textbf{Frobenius}. There is a rich literature on Gorenstein dg algebras dating back many years \cite{afgor, fjgorenstein, DGI, AIgorenstein, jin, goodbody}; our definition agrees with that of Avramov and Foxby. Moreover, we show that the Gorenstein and Frobenius properties are also Koszul dual. We also provide some rigorous justification of the heuristic that `smoothness is Koszul dual to properness'. Our main theorem states that the following properties are Koszul dual:

\begin{thmx}[\ref{mainthm}, \ref{kdpropc}, \ref{cogsmooth}]\label{thma}
    The following properties of augmented dg algebras and conilpotent dg coalgebras correspond across Koszul duality:
\begin{center}
\begin{tabular}{c | c}
 Property & Koszul dual property \\ \hline 
  (nonsmooth) Calabi--Yau & Symmetric \\  
 Gorenstein & Frobenius \\
  Smooth & Strongly proper local \\
  Regular & Proper local\\
\end{tabular}
\end{center}
\end{thmx}
As before, each row of this table comprises two theorems: for each property $P$, an augmented dg algebra $A$ has $P$ if and only if $BA$ has the Koszul dual property. We also prove twisted versions: namely, twisted Calabi--Yau is Koszul dual to twisted symmetric.

\p Many of our examples come from algebraic geometry, where we show that a nonsmooth Calabi--Yau scheme is singular Calabi--Yau in the sense of Iyama--Reiten (\ref{scy}). The converse holds in the local setting, where nonsmooth Calabi--Yau and Gorenstein are equivalent (\ref{gorcor}). We moreover show that Grothendieck duality for proper schemes $X$ can be reinterpreted as a Frobenius condition, which can be upgraded to a symmetric condition precisely when $X$ is nonsmooth CY (\ref{agex}).

\p Our main application is a new characterisation of Poincar\'e duality spaces. For a fixed field $k$ and pointed topological space $X$, we let $C_\bullet X$ be the dg coalgebra of $k$-chains on $X$ and $C^\bullet X$ the dg algebra of $k$-cochains. When $Y=\Omega X$ is a loop space, recall that concatenation of loops makes $C_\bullet Y$ into a dg algebra (in fact - up to quasi-isomorphism - a dg Hopf algebra).
\begin{thmx}[\ref{pdDict} and \ref{fhtthm}]\label{thmb}
    Let $X$ be a path connected finitely dominated topological space. Then the following are equivalent:
    \begin{enumerate}
        \item $X$ is a $d$-dimensional Poincar\'e duality space.
         \item $C_\bullet \Omega X$ is a $d$-Gorenstein algebra.
        \item $C_\bullet \Omega X$ is a $d$-Calabi--Yau algebra.
        \item $C_\bullet X$ is a $d$-Frobenius coalgebra.
        \item $C_\bullet X$ is a $d$-symmetric coalgebra.
    \end{enumerate}
    If in addition $X$ is simply connected, the above conditions are equivalent to the following:        \begin{enumerate}
        \item $C^\bullet X$ is a $(-d)$-Gorenstein algebra.
        \item $C^\bullet X$ is a $(-d)$-Frobenius algebra.
        \item $C^\bullet X$ is a $(-d)$-symmetric algebra.
        \item $C^\bullet X$ is a $(-d)$-Calabi--Yau algebra.
    \end{enumerate}
\end{thmx}
Our notion of Poincar\'e duality space is a fully derived version that gives chain-level Poincar\'e duality for all local systems of dg-$k$-vector spaces. In particular, it is sensitive to the base field: we warn that a space $X$ could be a Poincar\'e duality space for a given choice of $k$ and not for a different choice. Theorem \ref{thmb} generalises a characterisation of simply connected Poincar\'e duality spaces due to Felix, Halperin, and Thomas \cite{FHT}; namely that they are \textbf{Gorenstein spaces}. In fact we show that a Poincar\'e duality space is always a Gorenstein space, and that the converse is true for simply connected finite CW complexes (\ref{gspthm}).
\p In order to prove Theorem \ref{thmb}, we make use of the dg Hopf algebra structure on $C_\bullet \Omega X$ (strictly, one needs to take normalised chains on the Kan loop group of $\mathrm{Sing}X$). This use of the Hopf structure is key in order to translate between the one-sided conditions (Gorenstein, Frobenius) and the two-sided conditions (Calabi--Yau, symmetric). Specifically, we show that certain dg Hopf algebras are Calabi--Yau if and only they are Gorenstein (\ref{hopfthm}); the proof rests on a computation of their Hochschild cohomology, which generalises a theorem of Ginzburg and Kumar \cite{gkumar}.  We also give an application to unimodular Lie algebras (\ref{liethm}). As applications of our results on Poincar\'e duality, we also obtain some results in string topology, rational homotopy theory, and Lie group (co)homology.

\p We remark that all of our results which do not mention the bar construction remain true for non-conilpotent coalgebras. In this setting, one needs to use derived categories of the second kind on the algebra side to get a satisfactory module-comodule Koszul duality \cite{positselskitwokinds, GLDII}. Briefly, if $A$ is a dg algebra then its derived category of the second kind $D^\mathrm{II}(A)$ is a triangulated category constructed by localising the category of $A$-modules at a class of weak equivalences finer than the quasi-isomorphisms. Most relevant for us is that if $C$ is a coaugmented (non-conilpotent) dg coalgebra and $A\coloneqq \Omega C$ is its Koszul dual dg algebra, then there is a natural triangle equivalence  $D^\mathrm{II}(A) \simeq \dco(C)$. This ensures that in the non-conilpotent setting, our main theorems remain true \textit{mutatis mutandis} by replacing $D(A)$ by $D^\mathrm{II}(A)$. For example, our methods show that a non-conilpotent coaugmented dg coalgebra $C$ is symmetric if and only if $A\coloneqq \Omega C$ is \textbf{Calabi--Yau of the second kind}, meaning that there is a weak equivalence $\R\hom_{D^\mathrm{II}(A^e)}(A,A^e)\simeq A[n]$ in $D^\mathrm{II}(A^e)$.

\p In future work we hope to provide a similar analysis of the results of \cite{kellerCYC} on Calabi--Yau completions. Following our Poincar\'e duality results, we also hope to apply similar methodology to some results in string topology and the cohomology of $p$-compact groups. It would also be interesting to compare our results more fully with those of \cite{KTV, KTVII} on pre-Calabi--Yau structures, or to obtain similar results in the global setting of \cite{GKD}.

\p While preparing this paper, we learned of the similar results obtained by Holstein and Rivera \cite{HRCY}. They obtain analogous Koszul duality statements for left and right Calabi--Yau dg categories. There are three key differences between \cite{HRCY} and the current paper: first is that we drop any smoothness/properness hypotheses completely. This means that we are unable to obtain any results about (negative) cyclic homology, since in our setting CY structures are no longer given by Hochschild classes. However it means that we are able to get wider classes of results: for example, we show that the pseudocompact algebra $k\llbracket x \rrbracket$ is Frobenius (= weak right CY), a notion which is only defined in \cite{HRCY} for proper pseudocompact algebras. For example, our Theorem \ref{thma} generalises the weak one-object cases of \cite[3.26]{HRCY} and \cite[3.33]{HRCY}.

\p The second difference is that while \cite{HRCY} work with dg categories, we work only with dg algebras (i.e.\ one-object dg categories). Although we believe that our results probably extend to the many-object setting, we do not tackle such issues in the current paper. The dg categorical approach also removes the need for (co)augmentations.

\p The third difference is that we also obtain one-sided results for Gorenstein and Frobenius (co)algebras; while \cite{HRCY} do not consider these conditions their proofs would adapt in the relevant settings.

\p A key technical difference with the approach of \cite{HRCY} is our use of contramodules. Especially important for us is that we identify the linear dual functor on modules - across Koszul duality - with a certain sort of `derived contramodule dual' functor (\ref{onesd}). Holstein and Rivera also obtain a Poincar\'e duality result similar to ours, since the relevant dg algebras used in the proof are smooth.

\p Finally, we mention also that our results generalise some of those of the recent preprint \cite{mao}, which imposes connectivity conditions in order to reduce the study of proper dg coalgebras to proper dg algebras.

\p The authors would like to thank Benjamin Briggs, Timothy De Deyn, Callum Galvin,  Isambard Goodbody, John Greenlees, Julian Holstein, Sebastian Opper, David Pauksztello, Manuel Rivera, Alex Takeda, and Michael Wemyss for helpful conversations.

\subsection{Layout of the paper}

\p Section 2 contains preliminaries on algebras and coalgebras. We review co/contraderived categories, the five functor formalism for co/contramodules, and the co/contra correspondence.

\p Section 3 is a review of Positselski's Koszul duality and Guan--Holstein--Lazarev's bimodule Koszul duality. We also develop some dual functors for co/contramodules, which may be of independent interest.

\p Section 4 is a short section which contains some material on derived Picard groups. Having shown earlier in section 3 that bimodule Koszul duality is monoidal, we use this to show that the derived Picard group of an augmented algebra agrees with the derived Picard group of its Koszul dual coalgebra. We also provide a translation into contramodules.

\p Section 5 concerns notions of smoothness and properness for algebras and coalgebras. Loosely, we show that smoothness and properness are Koszul dual properties.

\p Section 6 concerns the nonsmooth CY condition on algebras. We also explore the links with CY structures on associated dg categories.

\p In section 7 we define Gorenstein and Frobenius algebras, and compare our notion of Gorenstein to others extant in the literature. Having developed the necessary technology in section 4, we show that a dg Frobenius algebra admits a Nakayama automorphism. We also discuss CY structures on some associated dg categories, and pay special attention to proper Frobenius algebras.

\p Section 8 contains the main theorems of this paper. We show that Gorenstein and Frobenius are Koszul dual concepts, and that nonsmooth CY and symmetric are similarly Koszul dual concepts. We also study finite dimensional Frobenius coalgebras in some detail, and give some results on endomorphism algebras that strengthen our earlier discussion of CY structures on categories.

\p Section 9, the final section, concerns applications to topology. We use Gorenstein duality and some results about dg Hopf algebras to obtain a new characterisation of Poincar\'e duality spaces. We give several examples.

\subsection{Notation and conventions}

Throughout this paper we work over a fixed base field $k$; for many purposes a commutative semisimple ring will do. Everything will be linear over $k$, and in particular unadorned tensor products are to be taken over $k$.

\p We denote isomorphisms with $\cong$ and weak equivalences with $\simeq$.

\p A \textbf{dg algebra} is a monoid in the category of dg-$k$-vector spaces, and a \textbf{dg coalgebra} is a comonoid. A dg algebra $A$ is \textbf{augmented} if the unit map $k\to A$ has a retract in the category of dg algebras. Similarly, a dg coalgebra is \textbf{coaugmented} if the counit map $C\to k$ has a section in the category of dg coalgebras. 

\p We will consistently drop the adjective `dg' with the convention that it is implicit, so for example a `vector space' is a dg vector space and an `algebra' is a dg algebra. We will use the modifier `discrete' to refer to objects concentrated in degree zero.

\p If $C$ is a coaugmented coalgebra, the coaugmentation coideal $\bar C \coloneqq C/k$ acquires a reduced comultiplication $\bar\Delta$. Say that $C$ is \textbf{conilpotent} if for every $c\in \bar C$, there exists some $n$ such that $\bar\Delta^n(c)=0$. By convention, every coalgebra we consider in this paper will be conilpotent; many of our theorems remain true without this assumption.

\p If $R$ is an algebra then $R^\circ$ will denote its opposite algebra and $R^e\coloneqq R^\circ \otimes R$ its enveloping algebra. An $R$-\textbf{module} will always mean a right $R$-module. Often we will refer to left $R$-modules as $R^\circ$-modules. An $R$-\textbf{bimodule} is a module over $R^e$. We use similar notation for coalgebras. Comodules will always be right comodules, and when we wish to consider left $C$-comodules we will usually simply consider comodules over the opposite coalgebra $C^\circ$.

\p  If $U$ is a vector space, we use $U^*$ to denote its $k$-linear dual. If $U,V$ are vector spaces, there is a natural map $U^*\otimes V^* \to (U\otimes V)^*$ given by $$f\otimes g \mapsto [ u\otimes v \mapsto (-1)^{|u||g|}f(u)g(v) ]$$and in particular we do not swap the factors in the tensor product. This ensures that if $C$ is a coalgebra, $C$-comodules dualise to $C^*$-modules.

\p We use cohomological indexing as standard; when discussing applications to topology we will occasionally use homological indexing. We denote shifts by $[1]$, so that if $V$ is a complex $V[1]$ denotes the complex with $V[1]^i = V^{i+1}$. In particular if $V$ is concentrated in degree $0$ then $V[1]$ is concentrated in degree $-1$.

\p If $\mathcal{T}$ is a triangulated category (or a pretriangulated dg category) and $S$ is a set of objects of $\mathcal{T}$, then $\thick(S)$ will denote the smallest full triangulated subcategory of $\mathcal{T}$ containing $S$ and closed under direct summands. It can be obtained as the closure of $S$ under shifts, cones and direct summands. Note that $\thick(S)$ is also closed under direct sums, as (up to a shift) these can be obtained as cones on the zero morphism.

\p If $R$ is an algebra, or more generally a dg category, we use $D(R)$ to denote the derived category of $R$. This is a pretriangulated dg category; often we will forget the dg structure and consider it simply as a triangulated category. The perfect derived category of $R$ will be denoted by $\per(R) \coloneqq {\thick}_{D(R)}(R)$; it is also a pretriangulated dg category which coincides with the compact objects of $D(R)$.

\p Let $R$ be an algebra. Say that $R$ is \textbf{proper} (or \textbf{perfectly valued}) if it has finite total cohomological dimension, i.e.\ $\oplus_nH^n(R)$ is a finite dimensional $k$-vector space. We use the same terminology for $R$-modules. If $\mathcal{A}$ is a dg category, its \textbf{perfectly valued derived category}, denoted $\pvd(\mathcal{A})$, is the subcategory of $D(R)$ on those modules which factor through $\per(k)\into D(k)$. When $R$ is an algebra, then $\pvd(R)$ consists precisely of those modules whose underlying dg vector spaces have finite dimensional total cohomology.

\p Say that an $R$-module $M$ is \textbf{reflexive} if the natural map $M\to M^{**}$ is a quasi-isomorphism. Clearly $M$ is reflexive if and only if each $H^iM$ is a finite dimensional vector space. Let $\rf(R)\subseteq D(R)$ denote the full triangulated subcategory on the reflexive modules. We have an obvious inclusion $\pvd(R)\into \rf(R)$; indeed a module is perfectly valued if and only if it is reflexive and bounded.

\section{Preliminaries on algebras and coalgebras}
We review some of the theory of comodules and contramodules over coalgebras, paying special attention to coderived and contraderived categories. For a more thorough account see Positselski's paper \cite{positselskitwokinds} or book \cite{PosBook}.

\subsection{Modules}

Let $R$ be an algebra. We regard the category $\Mod R$ of all $R$-modules as a model category with its usual projective model structure; its homotopy category is $D(R)$, the derived category of $R$. The functor $\hom_R(-,R):\Mod R \to \left(\Mod R^\circ\right)^\circ$ is right Quillen and hence admits a total derived functor $\R\hom_R(-,R): D(R) \to D(R^\circ)^\circ$ which we refer to as the \textbf{one-sided dual} functor or simply the \textbf{dual} when the context is clear. Clearly $\R\hom_R(R,R)\simeq R$, so by taking thick subcategories we get induced contravariant equivalences $$\R\hom_R(-,R):\per(R) \longleftrightarrow \per(R^\circ):\R\hom_{R^\circ}(-,R).$$

\p Similarly, if $R$ is an algebra and $M$ an $R$-module, we denote by $M^*$ its $k$-linear dual $\hom_k(M,k)$. Since this functor is exact, it is its own right derived functor. Linear duality hence defines a contravariant functor $D(R) \to D(R^\circ)$ which induces contravariant equivalences between $\rf(R)$ and $\rf(R^\circ)$.

\p If $R\simeq R^\circ$ as algebras, then we may regard both the one-sided and the $k$-linear duals as contravariant endofunctors on $D(R)$. In particular, suppose that $R=A^e$ for some other algebra $A$. If $M$ is an $A$-bimodule, then its \textbf{bimodule dual} is the $A$-bimodule $M^\vee \coloneqq \R\hom_{A^e}(M,A^e)$. The bimodule dual functor induces a contravariant autoequivalence on $\per(A^e)$. Similarly, the \textbf{linear dual} of $M$ is the $A$-bimodule $A^*$ and the linear dual functor is a contravariant autoequivalence of $\rf(A^e)$.

\p The category $\Mod A^e$ of $A$-bimodules is both right and left closed monoidal under the tensor product $\otimes_A$, with its adjoints given by $\hom_A$ and $\hom_{A^\circ}$. However it is not symmetric monoidal, unless $A$ is commutative. The above functors derive to give a non-symmetric monoidal structure on $D(A^e)$. For $M$ a fixed $A$-bimodule, the functor $M\lot_{A}-$ has a right adjoint, the \textbf{left internal hom} $\R\hom_{A^\circ}(M,-)$. Similarly, the functor $-\lot_{A}M$ has a right adjoint $\R\hom_{A}(M,-)$ which we call the \textbf{right internal hom}.

\p A bimodule $M$ is \textbf{symmetric} if $am=ma$, for all $a\in A$ and $m\in M$. Note that $A$ is commutative if and only if it is a symmetric bimodule over itself. When $A$ is commutative, then so is $A^e$, and the (derived) categories of $A$-modules, $A^\circ$-modules, and symmetric $A$-bimodules all coincide, and all are closed symmetric monoidal categories under $\lot_A$. Even when $A$ is commutative, the category of $A$-bimodules need not be symmetric monoidal (take two noncommuting automorphisms of $A$ and twist the regular bimodule by them).

\p As a consequence, if $A$ is a commutative algebra, then the category of all $A$-bimodules is equivalent to the category of symmetric $A^e$-bimodules. In particular, a symmetric $A$-bimodule is not the same thing as a symmetric $A^e$-bimodule.


\subsection{Comodules and contramodules}
Let $C$ be a coalgebra. A \textbf{comodule} over $C$ is a vector space $V$ with a coaction map $V\to V\otimes C$ which is compatible with the counit and comultiplication of $C$. These form a category $\Comod C$. The forgetful functor from $C$-comodules to vector spaces admits a right adjoint which sends $V$ to the cofree comodule $V\otimes C$. We regard the category $\Comod C$ of all $C$-comodules as a model category under Positselski's model structure \cite{positselskitwokinds}, where the weak equivalences are the morphisms with coacyclic cone and the cofibrations are the injections. The homotopy category is the \textbf{coderived category} which we denote by $\dco (C)$. It is a triangulated category, with shift functor given by the usual cohomological shift. We let $\fd(C)\into \dco(C)$ denote the full subcategory on the compact objects; an object of $\dco(C)$ is compact if and only if it is weakly equivalent to a finite dimensional $C$-comodule. Since all coalgebras are by assumption conilpotent, we have $\fd(C)\simeq \thick_{\dco(C)}(k)$, which is an equivalence we will use often.

\p A \textbf{contramodule} over $C$ is a vector space $V$ with a contraaction map $\hom(C,V) \to V$ that is compatible with the counit and comultiplication maps. These form a category $\Ctrmod C$. The forgetful functor from $C$-contramodules to vector spaces admits a left adjoint which sends $V$ to the free contramodule $\hom(C,V)$. The category $\Ctrmod C$ of all contramodules is also a model category \cite{positselskitwokinds}, where the weak equivalences are the morphisms with contraacyclic cone and the fibrations are the surjections. The homotopy category is the \textbf{contraderived category} which we denote by $\dct(C)$. It is also triangulated, with the same shift functor as before.

\p If $C,D$ are two coalgebras then a $C$-$D$-\textbf{bicomodule} is defined to be a $C^\circ \otimes D$-comodule. We denote the category of such by $C \bicomod D$. Similarly, a $C$-$D$-\textbf{bicontramodule} is a $C^\circ \otimes D$-contramodule, and they form a category $C \bictmod D$.

\p If $C$ is a coalgebra, its linear dual $C^*$ is an algebra under the convolution product $f.g= m(f\otimes g)\Delta$, where $m:k\otimes k \to k$ is the multiplication isomorphism. In fact, $C^*$ is a pseudocompact algebra, and the linear dual is a contravariant equivalence between the category of coalgebras and the category of pseudocompact algebras. Similarly, the linear dual gives a contravariant equivalence between the categories of $C$-comodules and pseudocompact $C^*$-modules. One can equivalently phrase all of our theorems about coalgebras in the language of pseudocompact algebras. In this paper we will often pass between the coalgebraic language and the pseudocompact language without comment.

\p The following is a key result. In its formulation, we use the notation ${}_CM_D$ to mean that $M$ is a $C$-$D$-bicomodule, and ${}^CX^D$ to mean that $X$ is a $C$-$D$-bicontramodule, with similar notation for one-sided co/contramodules.
\begin{thm}[Positselski]\label{fivefun}
Let $C,D,E,F$ be coalgebras. There are bifunctors 
    \begin{align*}
    -\square_C-&:\quad D\text{-}\comod C \quad\times\quad C\text{-}\comod E \quad \longrightarrow\quad D\text{-}\comod E\\
    -\odot_C-&:\quad \Ctrmod C \quad\times\quad C\bicomod D   \quad \longrightarrow\quad \comod D\\
    \cohom_C(-,-)&:\quad (D\bicomod C)^\circ \quad\times\quad E\bictmod C \longrightarrow\quad E\bictmod D\\
    \hom_C(-,-)&:\quad (D\bicomod C)^\circ \quad\times\quad \comod C \longrightarrow\quad \Ctrmod D\\
    \hom^C(-,-)&:\quad (\Ctrmod C)^\circ \quad\times\quad D\bictmod C \longrightarrow\quad D\text{-}\mathbf{Ctrmod} \end{align*}
satisfying natural isomorphisms, for comodules ${}_EM_C$, ${}_CN_D$, ${}_DP_F$, $Q_D$ and contramodules $X^E$, ${}^DY^C$:
\begin{align*}
M\square_C (N\square_D P) &\cong (M\square_C N)\square_D P  &\in E\bicomod F\\
         X\odot_E (M \square_C N) &\cong (X \odot_E M)\square_C N & \in \comod D\\
        \cohom_C(M,\hom_D(N,Q)) &\cong \hom_D(M\square_C N,Q)& \in \Ctrmod E\\
        \cohom_C(X\odot_EM,Y) &\cong \hom^E(X,\cohom_C(M,Y))&\in D\text{-}\mathbf{Ctrmod}
    \end{align*}
    We refer to $\square_C$ as the \textbf{cotensor product}, $\odot_C$ as the \textbf{contratensor product}, $\hom_C$ as the \textbf{comodule Hom}, and $\hom^C$ as the \textbf{contramodule Hom} or simply the \textbf{contrahom}. 
\end{thm}
We remark that for the associativity equations to hold it is necessary that our base ring be semisimple.


\begin{proof}
The existence of $\square_C$ is \cite[1.2.4]{PosBook} and associativity is \cite[1.2.5]{PosBook}. The existence of $\odot_C$ is \cite[5.1.1, 5.1.2]{PosBook} and the cotensor-contratensor associativity is \cite[5.2, Proposition 1]{PosBook}. The existence of $\hom_C$ is \cite[5.1.2]{PosBook}.

\p For brevity we will write $(U,V)\coloneqq \hom_k(U,V)$. Let $T$ be a $C$-comodule and $Z$ a $C$-contramodule. Recall that $\cohom_C(T,Z)$ is defined as the coequaliser of the natural diagram $(C\otimes T, Z) \rightrightarrows (T,Z)$ given by the $C$-coaction on $T$ and the $C$-contraaction on $Z$ \cite[3.2.1]{PosBook}. Suppose that $T$ is a $D$-$C$-bicomodule and $Z$ is an $E$-$C$-bicontramodule. We wish to construct a map $(D\otimes E,\cohom_C(T,Z)) \to \cohom_C(T,Z)$ exhibiting $\cohom_C(T,Z)$ as an $E$-$D$-bicontramodule. Since $\hom(D\otimes E,-)$ commutes with coequalisers, it suffices to show that the coequaliser diagram defining $\cohom_C(T,Z)$ is a diagram of $D$-$E$-bicontramodules. Let $\rho: T \to D\otimes T$ and $\sigma: (E,Z)\to Z$ denote the structure maps. Then the $D$-$E$-contraaction on $(T,Z)$ is given by the composition $$(D\otimes E,(T,Z)) \cong (D\otimes T,(E,Z)) \xrightarrow{(\rho,\sigma)}(T,Z)$$and similarly for $(C\otimes T,Z)$. Moreover the two maps in the coequaliser are $D$-$E$-bicontralinear, since they only use the $C$-co/contraactions. Hence $\cohom_C(T,Z)$ is a $D$-$E$-bicontramodule, as required. The cohom-cotensor adjunction is \cite[5.2, Proposition 2]{positselskitwokinds}.




\p Let $W,Y$ be $C$-contramodules, so that we have maps $(C,W)\to W$ and $(C,Y)\to Y$. The contrahom $\hom^C(W,Y)$ is defined as the equaliser of the natural pair of maps $(W,Y) \rightrightarrows ((C,W),Y)$ given by the contraactions on $W$ and $Y$ respectively. If $Y$ is in addition a $D$-contramodule, this diagram is a diagram of $D$-contramodules: indeed if $V$ is a vector space then as before $\hom(V,Y)$ is a $D$-contramodule under the map $$(D,(V,Y)) \cong (V,(D,Y)) \to (V,Y).$$Since $(D,-)$ preserves equalisers, $\hom^C(W,Y)$ is hence a $D$-contramodule. The cohom-contratensor adjunction is \cite[5.2, Proposition 3]{positselskitwokinds}.
\end{proof}
We next show that these functors are homotopically well-behaved. First recall that a \textbf{cofree} $C$-comodule is a comodule whose underlying graded $C$-comodule is of the form $V\otimes C$, with natural coaction. An \textbf{injective} $C$-comodule is a summand of a cofree comodule. Similarly, a \textbf{free} $C$-contramodule is a contramodule whose underlying graded $C$-contramodule is of the form $\hom(C,V)$, with natural contraaction, and a \textbf{projective} $C$-contramodule is a summand of a free contramodule. We denote the full subcategory of injective $C$-comodules by $\mathbf{Inj\text{-}}C$ and the full subcategory of projective $C$-contramodules by  $\mathbf{Proj\text{-}}C$. As before, if $D$ is another coalgebra, then we will abbreviate $\mathbf{Inj\text{-}}(C^\circ \otimes D)$ by $C\mathbf{\text{-}Inj\text{-}}D$ and $\mathbf{Proj\text{-}}(C^\circ \otimes D)$ by $C\mathbf{\text{-}Proj\text{-}}D$. Clearly an injective $C$-$D$-bicomodule is injective both as a $C$- and a $D$-bicomodule, and the same is true for projective bicontramodules. 

\begin{prop}\label{fivefunproj}
    Let $C,D,E$ be coalgebras. The bifunctors of \ref{fivefun} restrict to bifunctors 
    \begin{align*}
    -\square_C-&:\quad D\mathbf{\text{-}Inj\text{-}} C \quad\times\quad C\mathbf{\text{-}Inj\text{-}} E \quad \longrightarrow\quad D\mathbf{\text{-}Inj\text{-}} E\\
    -\odot_C-&:\quad \mathbf{Proj\text{-}} C \quad\times\quad C\mathbf{\text{-}Inj\text{-}} D   \quad \longrightarrow\quad \mathbf{Inj\text{-}} D\\
    \cohom_C(-,-)&:\quad (D\mathbf{\text{-}Inj\text{-}}C)^\circ \quad\times\quad E\mathbf{\text{-}Proj\text{-}} C \longrightarrow\quad E\mathbf{\text{-}Proj\text{-}} D\\
    \hom_C(-,-)&:\quad (D\mathbf{\text{-}Inj\text{-}} C)^\circ \quad\times\quad \mathbf{Inj\text{-}} C \longrightarrow\quad \mathbf{Proj\text{-}} D\\
    \hom^C(-,-)&:\quad (\mathbf{Proj\text{-}} C)^\circ \quad\times\quad D\mathbf{\text{-}Proj\text{-}} C \longrightarrow\quad D\mathbf{\text{-}Proj}. \end{align*}
\end{prop}

\begin{proof}
We begin with the cotensor product. If $M$ is an injective $D$-$C$-comodule and $N$ is an injective $C$-$E$-comodule, then $M\square_C N$ is a summand of a $D$-$E$-comodule of the form $(D\otimes W \otimes C)\square_C(C\otimes V \otimes E)\cong D\otimes W \otimes C\otimes V \otimes E$, and hence is an injective $D$-$E$-comodule. For the contratensor product, if $X$ is a projective $C$-contramodule and $M$ is an injective $C$-$D$-comodule, then $X\odot_C M$ is a summand of a $D$-comodule of the form $\hom(C,V)\odot_C M \cong V\otimes M$, which is an injective $D$-comodule. Hence $X\odot_C M$ is an injective $D$-comodule. For Cohom, if $M$ is an injective $D$-$C$-comodule and $X$ is a projective $E$-$C$-contramodule, then $\cohom_C(M,X)$ is a summand of an $E$-$D$-contramodule of the form 
\begin{align*}
    \cohom_C(D\otimes V \otimes C, \hom(E\otimes C,W))  &\cong  \cohom_C(D\otimes V \otimes C, \hom(C,\hom(E,W))) \\&\cong \hom(D\otimes V \otimes C,\hom(E,W))\\&\cong \hom(D\otimes E,\hom(V\otimes C,W))
\end{align*}which is free. Hence $\cohom_C(M,X)$ is a projective $E$-$D$-contramodule. For the comodule hom, if $M$ is an injective $D$-$C$-comodule and $N$ is an injective $C$-comodule, then $\hom_C(M,N)$ is a summand of a $D$-contramodule of the form
$$\hom_C(D\otimes W \otimes C,V\otimes C)\cong \hom(D\otimes W \otimes C,V)\cong \hom(D,\hom(W\otimes C,V))$$which is free, and hence $\hom_C(M,N)$ is projective. Finally, for the contramodule hom, if $X$ is a projective $C$-contramodule and $Y$ is a projective $D$-$C$-contramodule, then $\hom^C(X,Y)$ is a summand of a $D$-contramodule of the form $\hom^C(\hom(C,V),Y)\cong \hom(V,Y)$ which is projective, and hence $\hom^C(X,Y)$ is itself projective.
\end{proof}

\begin{thm}[Positselski]\label{derivedfive}
    All of the functors of \ref{fivefun} admit derived versions, and the derived versions of the associativity isomorphisms all hold.
\end{thm}
\begin{proof}
    Positselski \cite{PosBook} proves that all of the functors in question admit derived functors, defined by replacing every argument by an injective comodule or a projective contramodule as appropriate. The derived associativity isomorphisms follow from the underived versions plus \ref{fivefunproj}.
\end{proof}

\begin{rmk}
    Let $C$ be a coalgebra and $M,N$ two $C$-comodules. Then the vector space $\R\hom_C(M,N)$ is quasi-isomorphic to the dg categorical hom-complex $\dco(C)(M,N)$. Similarly, if $X,Y$ are two $C$-contramodules then $\R\hom^C(X,Y)$ is quasi-isomorphic to $\dct(C)(X,Y)$.
\end{rmk}

\begin{thm}[Co-contra correspondence {\cite{positselskitwokinds}}]
    Let $C$ be a coalgebra. Then the functors $-\odot_C C$ and $\hom_C(C,-)$ form a Quillen equivalence between $\comod C$ and $\Ctrmod C$. Moreover they are their own derived functors, and hence the same pair of functors give a triangle equivalence $\dco(C) \simeq \dct(C)$.
\end{thm}

\begin{rmk}
    The monoidal cotensor product on $\Comod C^e$ descends to a monoidal derived cotensor product on $\dco(C^e)$, and hence a monoidal structure on $\dct(C^e)$. However, this monoidal structure does not seem to admit a simple description at the underived level.
\end{rmk}

\begin{rmk}
    All of the results of this section remain true verbatim without the assumption of conilpotency.
\end{rmk}

\section{Koszul duality}
We review Positselski's formulation of Koszul duality and Guan--Holstein--Lazarev's formulation of bimodule Koszul duality. Using this, we define a linear dual functor for comodules and one- and two-sided duality functors for contramodules.

\p In this section all (co)algebras will be (co)augmented. For brevity in our theorems we will introduce the following notion. Say that a pair $(A,C)$ consisting of an algebra and a coalgebra is a \textbf{Koszul duality pair} if $C\simeq BA$, or equivalently $A\simeq\Omega C$, where $B$ and $\Omega$ denote the bar and cobar constructions respectively.

\p If $A$ is an algebra and $C$ a coalgebra, a morphism $\Omega C \simeq A$ defines a twisting cochain $C \to A$ by composition with the natural linear map $C \to \Omega C$. In particular, if $(A,C)$ is a Koszul duality pair then the quasi-isomorphism $\Omega C \simeq A$ defines a twisting cochain $\tau:C \to A$ which we call the \textbf{associated twisting cochain}.

\subsection{One-sided duality}
\begin{thm}[Positselski]\label{posKDthm}
 Let $(A,C)$ be a Koszul duality pair. 
 \begin{enumerate}
     \item There is a Quillen equivalence between $\Mod A$ and $\Comod C$.
     \item There is a Quillen equivalence between $\Mod A$ and $\Ctrmod C$.
     \item The Quillen equivalences of (1) and (2) are compatible with the co-contra correspondence, in the sense that we obtain a commutative diagram of quasi-equivalences of pretriangulated dg categories $$\begin{tikzcd}
         & D(A) \ar[dl]\ar[dr]& \\
         \dco(C) \ar[rr]&&\ar[ll] \dct(C)
     \end{tikzcd}$$
     \item The $A$-module $A\in D(A)$ corresponds to $k\in \dco(C)$.
     \item The $A$-module $k\in D(A)$ corresponds to $C\in \dco(C)$ and $C^*\in \dct(C)$.
     \item The $A$-module $A^* \in D(A)$ corresponds to $k\in \dct(C)$.
 \end{enumerate}
\end{thm}
\begin{proof}
    The existence of the model structures and Quillen functors is \cite[\S8]{positselskitwokinds}. Everything else is contained in \cite[\S6.5]{positselskitwokinds}.
\end{proof}

The Quillen functors between $\Mod A$ and $\Comod C$ are given by twisted tensor products, and the functors between $\Mod A$ and $\Ctrmod C$ are given by twisted Homs \cite[\S6.2]{positselskitwokinds}. Specifically, let $\tau:C \to A$ denote the associated twisting cochain. If $M$ is an $A$-module then the corresponding $C$-comodule is the twisted tensor product $M\ootimes C$; the underlying graded vector space is $M\otimes C$ and the differential is twisted by the term $m\otimes c \mapsto m\tau(c_1)\otimes c_2$, where we use Sweedler notation for the comultiplication of $C$. The coaction on $M\ootimes C$ is given by the natural coaction of $C$ on itself. Similarly, if $N$ is a $C$-comodule then the corresponding $A$-module is $N\ootimes A$, where the differential is twisted by the term $n\otimes a \mapsto n_1\otimes \tau(n_2)a$, where we again use Sweedler notation for the coaction $N\to N\otimes C$. Note that $M\ootimes C$ is a cofree comodule.

\p For contramodules the twisted Hom functors are similar: if $M$ is an $A$-module, the corresponding $C$-contramodule is given by the twisted Hom $\uhom(C,M)$. As a graded vector space, this agrees with $\hom(C,M)$, and the differential is twisted by the term which sends a function $g:C \to M$ to the function $c\mapsto g(c_1)\tau(c_2)$. The contraaction on $\uhom(C,M)$ is given by $\Delta^*$. Finally, if $X$ is a $C$-contramodule, the corresponding $A$-comodule is $\uhom(A,X)$. This time the differential is twisted by the map which sends a linear function $g:A \to X$ to the function $A \to X$ which sends $a\in A$ to the element $\pi\left[c\mapsto g(\tau(c)a) \right]$, where $\pi: \hom(C,X) \to X$ is the contraaction. More abstractly, this twisting term is the map defined by the composition $$\hom(A,X) \xrightarrow{(\mu\tau)^*} \hom(C\otimes A,X) \cong \hom(A,\hom(C,X))\xrightarrow{\pi_*}\hom(A,X)$$where $\mu\tau:C\otimes A \to A$ sends $c\otimes a$ to $\tau(c)a$. Note that $\uhom(C,M)$ is a free contramodule.
\begin{cor}
    If $(A,C)$ is a Koszul duality pair then there are algebra quasi-isomorphisms
    $$A\simeq \R\enn_{\dco(C)}(k)$$ 
    $$\R\enn_A(A^*) \simeq \R\enn_{\dct(C)}(k)$$
    $$C^*\simeq \R\enn_A(k)$$ 
\end{cor}
\begin{proof}
    The first two quasi-isomorphisms are a consequence of Koszul duality being a quasi-equivalence of dg categories. The third is similar: firstly, we obtain $\R\enn_A(k)\simeq \R\enn_{\dco(C)}(C)$. Because $C$ is an injective $C$-comodule, we may compute the derived endomorphism algebra as $\hom_C(C,C)\simeq C^*$.
\end{proof}

\begin{rmk}
    In particular, the hom-tensor adjunction gives us a vector space quasi-isomorphism $\R\enn_{\dct(C)}(k)\simeq A^{**}$.
\end{rmk}

\subsection{Bimodule duality}




Guan, Holstein and Lazarev showed that Koszul duality holds for bimodules:

\begin{thm}[\cite{bimodKD}]
Let $(A,C)$ and $(A',C')$ be Koszul duality pairs. Then $(A\otimes A',C\otimes C')$ is also a Koszul duality pair. 
\end{thm}
In particular, with notation as above there are induced triangle equivalences $D(A\bimod A'^\circ)\simeq \dco(C \bicomod C'^\circ) \simeq \dct(C \bictmod C'^\circ)$. This theorem is not a simple consequence of one-sided Koszul duality, since the bar and cobar constructions do not in general preserve the tensor product. 

\p Putting $C'=C^\circ$ in the above theorem we obtain:

\begin{cor}
If $(A,C)$ is a Koszul duality pair then so is $(A^e,C^e)$.
\end{cor}
In particular, if $(A,C)$ is a Koszul duality pair then there are triangle equivalences
$D(A^e) \simeq \dco(C^e)\simeq \dct(C^e)$.
\begin{rmk}
The above triangle equivalences enhance to quasi-equivalences of pretriangulated dg categories, and in \cite{bimodKD} this is used to compare the Hochschild theory of $A$ with that of the pseudocompact algebra $C^*$.
\end{rmk}

If $(A,C)$ is a Koszul duality pair and $M$ is an $A$-module then we denote by $M^!$ the corresponding $C$-comodule across the Koszul duality equivalence $D(A)\simeq \dco(C)$. Similarly if $N$ is a $C$-comodule then we let $N^!$ denote the corresponding $A$-module. We abuse this notation to also apply to bimodule Koszul duality.

\begin{prop}[Monoidality of bimodule Koszul duality]\label{kdmon}
Let $(A,C)$, $(A',C')$ and $(A'',C'')$ be Koszul duality pairs. Let $M$ be an $A'$-$A$-bimodule and let $N$ be an $A$-$A''$-bimodule. There is a natural weak equivalence
$$\left(M\lot_AN\right)^! \simeq M^! \square_C^{\mathbb{L}}N^!$$of $C'$-$C''$-bicomodules.
\end{prop}
\begin{proof}
Since bimodule Koszul duality is an equivalence of triangulated categories, it suffices to prove the equivalent statement that if $M,N$ are two bicomodules then there is a natural weak equivalence $(M\square_C^{\mathbb{L}}N)^!\simeq M^!\lot_A N^!$. To prove this, we may first assume that both $M$ and $N$ are injective as $C$-comodules. Since injective comodules are retracts of cofree comodules, it actually suffices to assume that both $M$ and $N$ are cofree $C$-comodules. A cofree bicomodule is in particular cofree on either side, so we may in fact assume that both $M$ and $N$ are cofree bicomodules. So put $M=C'\otimes U\otimes C$ and $N=C\otimes V\otimes C''$ for $U,V$ two vector spaces. We then have $M\square_C^{\mathbb{L}}N\simeq C'\otimes W\otimes C''$, where $W$ is the vector space $U\otimes C \otimes V$. Hence $(M\square_C^{\mathbb{L}}N)^!$ is the bimodule $C'^! \otimes W \otimes C''^!$. On the other hand, $M^! \otimes N^!$ is $C'^!\otimes U\otimes C^! \lot_A C^! \otimes V \otimes C''^!$. Hence it suffices to prove that the two vector spaces $W$ and $U\otimes C^! \lot_A C^! \otimes V$ are quasi-isomorphic. But to prove this it is enough to show that the vector spaces $C$ and $C^! \lot_A C^!$ are quasi-isomorphic. If $\tau$ denotes the associated twisting cochain for the pair $(A,C)$, then the derived tensor product is $C\otimes_\tau A \lot_A A\otimes_\tau C$. As a vector space, this simplifies to $C\otimes_\tau A \otimes C$, but since $C\otimes_\tau A \simeq k$, we obtain the desired quasi-isomorphism.
\end{proof}

\begin{cor}
    Let $(A,C)$ be a Koszul duality pair. Then there are quasi-isomorphisms $A\simeq k\square_C^{\mathbb{L}} k$ and $C\simeq k\lot_A k$ of vector spaces.
\end{cor}
\begin{proof}
    This follows from \ref{kdmon} by putting $A'=A''=k$, $C'=C''=k$, and $M=N=A$ and $M=N=k$ respectively.
\end{proof}

The following is a key computation:

\begin{prop}\label{diagonal}
    Let $(A,C)$ be a Koszul duality pair. Across the derived equivalence $D(A^e)\simeq \dco(C^e)$, the diagonal bimodule $A$ corresponds to the diagonal bicomodule $C$.
\end{prop}
\begin{proof}
    By \ref{kdmon}, if $N$ is a $C$-bicomodule then there is a natural weak equivalence $A^!\square_C^{\mathbb{L}}N\simeq N$. The claim follows by putting $N=C$.
\end{proof}

\begin{cor}\label{codiagonal}
    Let $(A,C)$ be a Koszul duality pair. Across the derived equivalence $D(A^e)\simeq \dct(C^e)$, the bimodule $A^*$ corresponds to the diagonal bicontramodule $C^*$.
\end{cor}
\begin{proof}
Across Koszul duality, the bimodule $A^*$ corresponds to the bicontramodule $\uhom(C^e,A^*)$. The hom-tensor adjunction gives an isomorphism of graded vector spaces $\hom(C^e,A^*)\cong \left(C^e\otimes A\right)^*$, and twisting the differentials gives an isomorphism $\uhom(C^e,A^*)\cong \left(C^e\ootimes A\right)^*$ of vector spaces. Moreover, the $C^e$-contramodule structure on both sides comes from the $C^e$-coaction on itself, and hence we have an isomorphism $\uhom(C^e,A^*)\cong \left(C^e\ootimes A\right)^*$ of $C$-bicontramodules. However, \ref{diagonal} shows that the bicomodule $C^e\ootimes A$ is weakly equivalent to the bicomodule $C$. Hence $\uhom(C^e,A^*)$ is weakly equivalent to $C^*$, as desired.
\end{proof}

\subsection{Linear duals for comodules}

Let $C$ be a coalgebra and $M$ a right $C$-comodule. If $M$ is finite dimensional then the linear dual $M^*$ is a left $C$-comodule: a coaction map $M\to M\otimes C$ is equivalent to a map $M\otimes M^*\to C$, which is in turn equivalent to a map $M^* \to C\otimes M^*$ making $M^*$ into a comodule.

\begin{prop}\label{GalgFcog}
If $(A,C)$ is a Koszul duality pair then there is a commutative diagram of functors $$\begin{tikzcd}[column sep = 1in] D(A)^\circ \ar[r,"{\R\hom_A(-,A)}"] \ar[d,"\simeq"]& D(A^\circ)\ar[d,"\simeq"]\\ \dco (C)^\circ \ar[r,"(-)^*"]& \dco (C^\circ)\end{tikzcd}$$Moreover, on objects the functor $(-)^*$ can be described as follows. If $M$ is a $C$-comodule, write $M=\varinjlim_i M_i$ where the $M_i$ range over the finite dimensional subcomodules of $M$. Then $M^*\coloneqq \holim_i M_i^*$.
\end{prop}
\begin{proof}
Consider the composition $$F:\dco(C)^\circ \xrightarrow{\simeq} D(A)^\circ \xrightarrow{\R\hom_A(-,A)} D(A^\circ) \xrightarrow{\simeq} \dco(C^\circ)$$which manifestly makes the diagram commute. Since $\R\hom_A(-,A)$ preserves homotopy limits, so does $F$. Observing that the colimit $M=\varinjlim_i M_i$ is a homotopy limit in $\dco(C)^\circ$, it remains to show that for finite dimensional comodules the functor $F$ agrees with the linear dual. Since $\R\hom_A(A,A)\simeq A$ we have $F(k) \cong k\cong k^*$. Since we have $\fd(C)\simeq \thick_{\dco(C)}$(k), it suffices to check that the induced map $F_{kk}:\R\enn_C(k) \to \R\enn_{C^\circ}(k)^\circ$ on derived endomorphisms is a quasi-isomorphism. But across Koszul duality this is identified with $\id_A$.
\end{proof}

We refer to $M^*$ as the \textbf{linear dual} of $M$. We warn that $M^*$ depends on the $C$-comodule structure of $M$, and not just on the underlying vector space, since limits in $C$-comodules are not computed by the underlying vector spaces. For a concrete example of this, see \ref{funnydual} below.

Observe that, across the contravariant equivalence between comodules and pseudocompact modules, the above also gives a linear dual functor for pseudocompact modules. Concretely, if $A$ is a pseudocompact algebra and $M$ a pseudocompact $A$-module, then we can write $M\cong \varprojlim_i M_i$ with each $M_i$ finite dimensional. Then we have $M^*\coloneqq \hocolim_i M_i^*$, where the homotopy colimit is taken in the category of pseudocompact modules.

\begin{ex}\label{funnydual}
	Consider the pseudocompact algebra $A\coloneqq k\llbracket t \rrbracket$ with $t$ in degree zero, equipped with its usual inverse limit topology. We will compute $A^*\simeq A[1]$. To do this, put $A\cong\varprojlim_n A_n$ where $A_n\coloneqq A/t^n$. We compute $A_n^*\simeq A_n$: explicitly, if we take the basis $\{t^0,\ldots, t^{n-1}\}$ for $A_n$, with corresponding dual basis $\{\delta_{t^0},\ldots,\delta_{t^{n-1}}\}$, then the $k$-linear map $t^i\mapsto \delta_{t^{n-i-1}}$ is an $A$-linear isomorphism. The projection map $A_{n+1} \to A_n$ induces the identity on functionals, and hence across the above isomorphism corresponds to multiplication by $t$. Hence $A^*$ is the homotopy colimit in pseudocompact $A$-modules of the system $$A_1 \xrightarrow{t} A_2 \xrightarrow{t} A_3 \xrightarrow{t}\cdots$$We can resolve each $A_n$ by the two-term complex $A\xrightarrow{t^n}A$, with the rightmost $A$ placed in degree zero. The above homotopy colimit becomes the colimit of the system of pseudocompact $A$-modules $$\begin{tikzcd}
		A \ar[r,"\id"]\ar[d,"t"]& A\ar[r,"\id"]\ar[d,"t^2"] & A\ar[r,"\id"]\ar[d,"t^3"] & \cdots\\
		A \ar[r,"t"]& A\ar[r,"t"] & A\ar[r,"t"] & \cdots
	\end{tikzcd}$$The colimit of the degree $-1$ part is obviously just $A$. The colimit of the degree zero part, taken in discrete $A$-modules, is $A[t^{-1}]$. The pseudocompactification of this module is zero, and hence the whole colimit is $A[1]$, as required. Since $A$ is Koszul dual to the discrete algebra $R\coloneqq k[\varepsilon]/\varepsilon^2$ with $\varepsilon$ in cohomological degree $1$, the above calculation is Koszul dual to the fact that $\R\hom_R(k,R)\simeq k[-1]$. More generally, if $t$ is placed in degree $n$, the same calculation shows that $A^*\simeq A[1-n]$. We remark that this calculation shows that $A$ is a $(1-n)$-Frobenius coalgebra (see \ref{cogmaindef} for the definition). This computation is generalised in \ref{pwfrobcor}.
\end{ex}

\begin{rmk}
    The `na\"i\kern -1pt ve' linear dual functor is well defined on the full subcategory $\fd(C)\into \dco(C)$ of compact objects, and the full linear dual functor of \ref{GalgFcog} can be obtained as a homotopy Kan extension of the na\"i\kern -1pt ve linear dual along the inclusion.
\end{rmk}

If $M$ is a $C^e$-comodule we denote the $C^e$-comodule $M^*$ by $D(M)$, to avoid confusion between one- and two-sided linear duals; as before the underlying $C^\circ$-comodule of $D(M)$ need not be $M^*$. Once again, we extend this definition to pseudocompact bimodules in the obvious manner.
\begin{ex}
    The calculation \ref{funnydual} shows that if $A$ is the pseudocompact algebra $k\llbracket t \rrbracket$ with $t$ placed in degree $n$, then $D(A)\simeq A[1-n]$ as pseudocompact bimodules. Hence $A$ is a $(1-n)$-symmetric coalgebra (see \ref{cogmaindef} for the definition).
\end{ex}

\begin{prop}\label{CYalgSFcog}
If $(A,C)$ is a Koszul duality pair then there is a commutative diagram of functors $$\begin{tikzcd}[column sep = 1in] D(A^e)^\circ \ar[r,"{\R\hom_{A^e}(-,A^e)}"] \ar[d,"\simeq"]& D(A^e)\ar[d,"\simeq"]\\ \dco (C^e)^\circ \ar[r,"D"]& \dco (C^e).\end{tikzcd}$$
\end{prop}
\begin{proof}
By \ref{GalgFcog} we have a commutative diagram $$\begin{tikzcd}[column sep = 1in] D(A^e)^\circ \ar[r,"{\R\hom_{A^e}(-,A^e)}"] \ar[d,"\simeq"]& D(A^e)\ar[d,"\simeq"]\\ \dco (B(A^e))^\circ \ar[r,"(-)^*"]& \dco (B(A^e)^\circ).\end{tikzcd}$$So we need only check that the diagram $$\begin{tikzcd}[column sep = 1in] \dco (B(A^e))^\circ\ar[d,"\simeq"] \ar[r,"(-)^*"]& \dco (B(A^e)^\circ)\ar[d,"\simeq"]\\
\dco (C^e)^\circ \ar[r,"D"]& \dco (C^e).\end{tikzcd}$$commutes, where the vertical equivalences are given by bimodule Koszul duality. To see this latter claim, proceed as in the proof of \ref{GalgFcog}: it suffices to check the claim on the comodule $k$, which is easy.
\end{proof}

\begin{rmk}
Since $\R\hom_A(-,A)$ induces a contravariant equivalence between $\per A$ and $\per A^\circ$, the linear dual functor induces a contravariant equivalence between $\thick_{\dco (C)}(k)$ and $\thick_{\dco(C^\circ)}(k)$. Since the category $\thick_{\dco (C)}(k)$ agrees with the category $\fd(C)$ of all compact $C$-comodules, the linear dual functor on $\dco C$ restricts to a contravariant equivalence between $\fd (C)$ and $\fd(C^\circ)$.
\end{rmk}

\begin{rmk}
Let $M$ be a $C$-bicomodule and let $M^*$ denote the linear dual of the $C$-comodule $M$. Let $J$ denote the diagram of sub-$C$-bicomodules of $M$ and let $J^*$ denote the diagram of $C$-bicomodules obtained by taking the linear dual levelwise. Since every finite dimensional sub-$C$-comodule of $M$ is contained in a sub-$C$-bicomodule, the diagram $J$ is cofinal in the diagram of sub-$C$-comodules of $M$. If $U$ denotes the forgetful functor from $C$-bicomodules to $C$-comodules, then we have $D(M)\simeq \holim J^*$ whereas $M^* \simeq \holim UJ^*$. Since $U$ does not preserve (homotopy) limits, the underlying $C$-comodule of $D(M)$ need not be $M^*$. For a related example, consider the pseudocompact dg algebras $A\coloneqq k$ and $A'\coloneqq k\llbracket x \rrbracket$, with $x$ in degree zero, and let $M$ be the pseudocompact $A\hat\otimes A'$-bimodule $A'$. The linear dual of $M$ in the category of pseudocompact $A\hat\otimes A'$-bimodules is $M[1]$, as \ref{funnydual} shows. On the other hand, the linear dual of $M$, taken in the category of pseudocompact $A$-bimodules, is the vector space 
$$\frac{k \llbracket x \rrbracket[x^{-1}]}{k\llbracket x \rrbracket}$$concentrated in degree zero; this can be extracted from the proof of \ref{funnydual} together with the fact that the pseudocompact derived category of $k$ is simply the usual derived category of vector spaces.
\end{rmk}

\subsection{One- and two-sided duals for contramodules}

Our main goal in this section is to prove the following theorem, which may be of independent interest:

\begin{thm}[One-sided duality]\label{onesd}
    Let $(A,C)$ be a Koszul duality pair. Then there is a commutative diagram of functors
    $$\begin{tikzcd}
     {\dct(C)}^\circ \ar[rr,"{\R\hom^C(-,C^*)}"]\ar[d,"\simeq"] &&  {\dct(C^\circ)}\ar[d,"\simeq"]\\
        {D(A)}^\circ \ar[rr,"{M\mapsto M^*}"] &&  {D(A^\circ)}.        
    \end{tikzcd}$$
\end{thm}
\begin{proof}
Recall that if $M$ is an $A$-module, the corresponding $C$-contramodule is the twisted Hom, denoted by $\uhom(C,M)$.
It suffices to show that  we have natural weak equivalences of contramodules $$  \R\hom^C(\uhom(C,M),C^*) \simeq \uhom(C,M^*).$$Since $\uhom(C,M)$ is a projective contramodule, we can compute $$\R\hom^C(\uhom(C,M),C^*) \simeq \hom^C(\uhom(C,M),C^*).$$ Recall that if $V$ is a vector space and $X$ is a $C$-contramodule, then the natural isomorphism $\hom^C(\hom(C,V),X) \to \hom(V,X)$ sends a $C$-contralinear map $F$ to the function $v\mapsto F(v\eta)$, where $\eta:C \to k$ is the counit and $v:k \to V$ denotes the function that sends $1\in k$ to $v$. Hence using the hom-tensor adjunction we obtain isomorphisms of vector spaces \begin{align*}
    \hom^C(\hom(C,M),C^*)&\cong \hom(M,C^*)\\
    &\cong \left(M\otimes C\right)^*\\
     &\cong \hom(C,M^*)
\end{align*}
and one can check that the composite isomorphism $\hom^C(\hom(C,M),C^*) \to \hom(C,M^*)$ sends a $C$-contralinear map $F$ to the map $C \to M^*$ which sends $c$ to the map $m\mapsto F(m\eta)(c)$. This gives us an isomorphism of graded vector spaces $$\Xi:\hom^C(\uhom(C,M),C^*)^\sharp \to \uhom(C,M^*)^\sharp$$and so we simply need to check compatibility of $\Xi$ with the differentials. The source of $\Xi$ acquires a twist from the $A$-module structure on $M$: the differential at $F$ is twisted by the term $F\phi$, where $\phi$ is the endomorphism of $\hom(C,M)$ which sends a function $g$ to $c\mapsto g(c_1)\tau(c_2)$, where $\tau$ is the associated twisting cochain. On the target of $\Xi$, the differential is twisted by the analogous map $\psi$, which sends a map $G$ to the map $c\mapsto G(c_1)\tau(c_2)$. So we need to check that $$\Xi(F\phi) = \psi(\Xi F)$$which follows since both are the map that sends $c$ to the map $m\mapsto \left(F[(m\tau(c_2))\eta]\right)(c_1)$.
\end{proof}

\begin{rmk}\label{hqfrmk}
Say that a $C$-comodule $M$ is \textbf{homotopy quasi-finite} (\textbf{hqf} for short) if, for all finite dimensional comodules $N$, the vector space $\R\hom_C(N,M)$ is an element of $\pvd(k)$. This is a homotopical version of Takeuchi's definition of a quasi-finite comodule \cite{takeuchi}. Let $\hqfco(C)$ denote the full subcategory of $\dco(C)$ on the hqf comodules. Note that we have an equivalence $\hqfco(C)\simeq \pvd\left(\dco(C)\right)$. If $(A,C)$ is a Koszul duality pair, then across Koszul duality the category $\hqfco(C)$ corresponds to $\pvd(A)$: this is because finite dimensional comodules correspond to perfect $A$-modules, and $V\in \pvd(A)$ if and only if $\R\hom_A(\per(A),V) \in \pvd(k)$. Similarly, say that a contramodule $X$ is hqf precisely when the corresponding comodule $X\odot_CC$ is; the hqf contramodules form a full subcategory $\hqfct(C)$ of $ \dct(C)$ that can be identified with $\pvd\dct(C)$. Since the linear dual is an equivalence $\pvd(A)^\circ \to \pvd(A^\circ)$, by \ref{onesd} it follows that the functor $\R\hom^C(-,C^*)$ gives a contravariant equivalence between $\hqfct(C)^\circ$ and $\hqfct(C^\circ)$.
\end{rmk}

\begin{prop}[Two-sided duality]\label{twosd}
    Let $(A,C)$ be a Koszul duality pair. If $R=(C^e)^*$ denotes the regular $C$-bicontramodule, then there is a commutative diagram of functors
     $$\begin{tikzcd}
     {\dct(C^e)}^\circ \ar[rr,"{\R\hom^{C^e}(-,R)}"]\ar[d,"\simeq"] &&  {\dct(C^{e})}\ar[d,"\simeq"]\\
        {D(A^e)}^\circ \ar[rr,"{M\mapsto M^*}"] &&  {D(A^{e})}.        
    \end{tikzcd}$$
\end{prop}
\begin{proof}
Letting $A'$ be the algebra $\Omega(C^e)$, the Koszul duality pair $(C^e,A')$ gives by \ref{onesd} a commutative diagram
 $$\begin{tikzcd}
     {\dct(C^e)}^\circ \ar[rr,"{\R\hom^{C^e}(-,R)}"]\ar[d,"\simeq"] &&  {\dct(C^{e\circ})}\ar[d,"\simeq"]\\
        {D(A')}^\circ \ar[rr,"{M\mapsto M^*}"] &&  {D(A'^\circ)}.        
    \end{tikzcd}$$Bimodule Koszul duality furthermore yields a commutative diagram  $$\begin{tikzcd}
    {D(A')}^\circ \ar[d,"\simeq"] \ar[rr,"{M\mapsto M^*}"] &&  {D(A'^\circ)}\ar[d, "\simeq"]\\
    {D(A^e)}^\circ \ar[rr,"{M\mapsto M^*}"] &&  {D(A^{e\circ})}
    \end{tikzcd}$$and hence pasting these together we obtain a commutative diagram
    $$\begin{tikzcd}
     {\dct(C^e)}^\circ \ar[rr,"{\R\hom^{C^e}(-,R)}"]\ar[d,"\simeq"] &&  {\dct(C^{e\circ})}\ar[d,"\simeq"]\\
        {D(A^e)}^\circ \ar[rr,"{M\mapsto M^*}"] &&  {D(A^{e\circ})}.        
    \end{tikzcd}$$Since both $C^e$ and $A^e$ are isomorphic to their own opposites, identifying the right-hand vertical map with the natural equivalence $\dct(C^e) \to D(A^e)$ yields the desired result.
\end{proof}

\section{Invertible modules and twists}
The goal of this section is to record some results about invertible bimodules, which we will later use to show that derived Frobenius algebras admit Nakayama automorphisms. We define derived Picard groups for algebras and coalgebras, and show that they are preserved across Koszul duality. We pay special attention to homotopy automorphisms. Derived Picard groups have been well studied in the literature, including but not limited to the papers \cite{rzpic, yekpic, kelpic, opper}; our treatment here is quite elementary.

\subsection{Derived Picard groups}

Let $A$ be an algebra. An $A$-bimodule $X$ is \textbf{invertible} if there exists another $A$-bimodule $Y$ such that $X\lot_AY\simeq A\simeq Y\lot_A X$. We write $Y=X^{-1}$, with the understanding that $Y$ is only unique up to a quasi-isomorphism of bimodules.

\p If $X$ is invertible, then the functor $F_X\coloneqq -\lot_A X$ is an autoequivalence of $D(A)$, with inverse $-\lot_A X^{-1}$. On the other hand the inverse of $F_X$ is necessarily its adjoint $\R\hom_A(X,-)$. This shows that $\R\hom_A(X,-)$ commutes with direct sums, and hence that $X$ is perfect as a right $A$-module; a similar argument shows that $X$ is also a perfect left $A$-module. Composing $F_X$ with its inverse yields natural quasi-isomorphisms $A\simeq \R\enn_A(X)$ and $\R\hom_A(X,A)\lot_A X \simeq A$, and a similar argument with left modules shows that in fact $X^{-1}\simeq \R\hom_A(X,A)$.

\p The \textbf{derived Picard group} $\mathrm{DPic}(A)$ is the group whose elements are the quasi-isomorphism classes of invertible $A$-bimodules, with group operation given by $\lot_A$. The identity element is $A$, the diagonal bimodule. The assignment $X\mapsto X\lot_A -$ defines a group morphism from $\mathrm{DPic}(A)$ to $\mathrm{Aut}(D(A))$, the group of exact autoequivalences of $D(A)$. In the future we will write ${}_XM\coloneqq X\lot_AM$ and $M_X\coloneqq M\lot_A X$ for brevity.

\begin{prop}\label{pregor}
    Let $A$ be an algebra and $X$ an $A$-bimodule. The following are equivalent:
    \begin{enumerate}
        \item $X$ is invertible.
        \item The functor $-\lot_A X$ is an autoequivalence of $\per A$.
        \item The functor $X\lot_A -$ is an autoequivalence of $\per (A^\circ)$.
        \item The right module $X$ is a thick generator of $\per A$ and the natural map $A \to \R\enn_A(X)$ is a quasi-isomorphism.
        \item The left module $X$ is a thick generator of $\per (A^\circ)$ and the natural map $A^\circ \to \R\enn_{A^\circ}(X)$ is a quasi-isomorphism.
    \end{enumerate}
\end{prop}
\begin{proof}
We show that (1), (2), and (4) are equivalent; the other cases are similar. We have already observed that (1) implies (2). For the converse, if $-\lot_A X$ is an autoequivalence then its inverse is necessarily given by its adjoint $-\lot_A \R\hom_A(X,A)$ and it now follows that $X$ is invertible. If (2) holds then since the autoequivalence $-\lot_A X$ takes generators to generators, $X$ is a generator of $\per A$. We have already observe that for an invertible module, the natural map is a quasi-isomorphism, so (4) holds. Conversely, if (4) holds then the functor $-\lot_A X$ gives an equivalence between $\per(A)$ and $\per(\R\enn_A(X)) \simeq \thick_A(X)\simeq \per(A)$.
\qedhere





\end{proof}

If $C$ is a coalgebra, say that a $C$-bicomodule $X$ is \textbf{invertible} if there exists another bicomodule $Y$ such that $X\square^\mathbb{L}_C Y \simeq C \simeq Y\square^\mathbb{L}_C X$. As above, we define the \textbf{derived Picard group} of $C$ to be the group $\mathrm{DPic}(C)$ whose elements are the weak equivalence classes of invertible bicomodules, with group operation given by the derived cotensor product. As before, if $X$ is an invertible bicomodule and $N$ is any comodule (left or right as appropriate) then we write $N_X\coloneqq N\square^\mathbb{L}_C X$ and ${}_XN \coloneqq X\square^\mathbb{L}_C N$.

\p By \ref{kdmon}, if $(A,C)$ is a Koszul duality pair then the two monoidal triangulated categories $D(A^e)$ and $\dco(C^e)$ are equivalent. We immediately obtain the following:

\begin{prop}
    If $(A,C)$ is a Koszul duality pair then there is an isomorphism $\mathrm{DPic}(A)\cong \mathrm{DPic}(C)$.
\end{prop}

\begin{ex}
If $A$ is an augmented algebra, then an \textbf{$\infty$-automorphism} of $A$ is a coalgebra automorphism $\theta$ of $C\coloneqq BA$. If $C_\theta$ denotes the $C$-bicomodule $C$, where the left coaction is as usual and the right coaction is twisted by $\theta$, then $C_\theta$ is an invertible bicomodule (with inverse $C_{\theta^{-1}}$). Hence $\theta$ defines a class $C_\theta \in \mathrm{DPic}(C) \cong \mathrm{DPic}(A)$.
\end{ex}
We can also transport the action of $\mathrm{DPic}(C)$ on $\dco(C)$ across the co-contra correspondence:
\begin{defn}
    Let $C$ be a coalgebra and $N$ a $C$-contramodule. If $X$ is an invertible $C$-bicomodule we put $N_{X}\coloneqq \R\cohom_C(X^{-1},N)$.
\end{defn}
\begin{prop}
    Let $C$ be a coalgebra, $M$ a $C$-comodule, and $N$ a $C$-contramodule. Then there are natural weak equivalences
    $$\hom_C(C,M_X) \simeq \hom_C(C,M)_X$$
    $$(N_X)\odot C \simeq (N\odot C)_X$$
of contramodules and comodules respectively.
\end{prop}
\begin{proof}
By the co-contra correspondence, $-\odot C$ and $\hom_C(C,-)$ are inverse equivalences, so the first and second equations are equivalent. We focus on the first. Letting $Y$ be the inverse of $X$, the dg functors $-\square^\mathbb{L}_CX$ and $-\square^\mathbb{L}_CY$ are inverse, which gives us a natural weak equivalence $\hom_C(C,M_X)\simeq \R\hom_C(Y,M)$. By adjunction we have natural weak equivalences $$\hom_C(C,M)_X\coloneqq \R\cohom_C(Y,\hom_C(C,M))\simeq \R\hom_C(Y,M)$$ and we are done.
\end{proof}

\begin{rmk}
    One can say that a $C$-bicontramodule $X$ is invertible precisely when its corresponding $C$-bicomodule $X\odot_{C^e}C^e$ is. The unit invertible bicontramodule can then be written as $U=\R\hom_{C^e}(C^e,C)$.
\end{rmk}

\subsection{Twists}

Recall that if $A$ is an algebra, a \textbf{homotopy endomorphism} of $A$ is an endomorphism in the homotopy category of algebras and a \textbf{homotopy automorphism} of $A$ is an automorphism of $A$ in the homotopy category of algebras. A homotopy endomorphism of $A$ can be represented by a map $\phi:\tilde A \to A$ where $\tilde A$ is a cofibrant algebra resolution of $A$. Given such a $\phi$, one obtains an endofunctor of $D(A^e)$ which sends $M$ to $M\lot_{\tilde A}A$. We denote the image of $A$ under this functor by $A_\phi$, so that $M\lot_{\tilde A}A\simeq M\lot_A A_\phi$. We abbreviate this module by $M_\phi \coloneqq M\lot_A A_\phi$ and call it the \textbf{right twist of $M$ by $\phi$}. If $\phi$ was a quasi-isomorphism, then the induced endofunctor $\phi$ is an autoequivalence. Similarly, via considering the functor $M \mapsto A\lot_{\tilde A}M$ one gets a \textbf{left twist} ${}_\phi M$ of $M$.

\p In particular, if $\phi$ is a homotopy automorphism of $A$, then ${}_\phi A$ is an invertible bimodule, with inverse $_{\phi^{-1}}A$. The assignment $\phi \mapsto {}_\phi A$ determines a group morphism $\mathrm{hAut}(A) \to \mathrm{DPic}(A)$, where $\mathrm{hAut}(A) $ is the group of automorphisms of $A$ in the homotopy category of algebras. More generally, there is a monoid morphism from $\mathrm{hEnd}(A)$ to the monoid of quasi-isomorphism classes of $A$-bimodules under the derived tensor product.

\begin{ex}
    Let $\phi:A \to A$ be an algebra endomorphism of $A$, and regard it as a homotopy endomorphism. Then the twist $M_\phi$ is (quasi-isomorphic to) the $A$-bimodule whose underlying vector space is $M$, and whose right $A$-action is twisted by the action of $\phi$. 
\end{ex}

The following proposition will allow us to deduce the existence of Nakayama automorphisms for derived Frobenius algebras (for the discrete version, see \cite[\S16E]{lam}).

\begin{prop}\label{nakesistence}
Let $A$ be an algebra. Let $M$ be an $A$-bimodule with a quasi-isomorphism $M\simeq A$ of right $A$-modules. Then there exists a homotopy endomorphism $\phi$ of $A$ such that $M \simeq {}_\phi A$ as $A$-bimodules. The map $\phi$ is a homotopy automorphism if and only if the $A$-action $A\to \R\enn_A(M)$ is a quasi-isomorphism.
\end{prop}

\begin{proof}
Pick an $A$-bimodule cofibrant replacement $\tilde M$ of $M$, so that there are homotopy inverse maps $f:A \to \tilde M$ and $g: \tilde M \to A$. These induce a quasi-isomorphism $q:\enn_A(A)\cong A \to \enn_A(\tilde M)$ of endomorphism algebras, and viewing $\tilde M$ as a left $A$-module via $q$ makes $f$ into an $A$-bilinear quasi-isomorphism. The left $A$-action $\rho$ on $\tilde M$ hence gives a diagram $A \xrightarrow{\rho} \enn(\tilde M) \xleftarrow{q} A$. This diagram gives a homotopy endomorphism $\phi$ of $A$ such that $\phi q\simeq \rho $ as morphisms in the homotopy category of algebras. It follows that $f$ is an $A$-bilinear quasi-isomorphism from ${}_\phi A$ to $M$, as required. Since $q$ is a quasi-isomorphism, $\rho$ is a quasi-isomorphism if and only if $\phi$ is a homotopy automorphism.
\end{proof}

\begin{rmk}\label{augnakrmk}
Suppose that $A\to k$ is an augmented algebra, so that one has $A\simeq k\oplus \bar A$ as $A$-bimodules. Hence the same holds for $M$ as a right module, which makes $\R\enn_A(M)$, and hence $\enn_A(\tilde M)$, an augmented algebra. If one chooses $f$ and $g$ to respect these direct sum decompositions, then $q$ becomes a morphism of augmented algebras. If $\rho$ is a map of augmented algebras, then one can choose $\phi$ to be a morphism in the homotopy category of augmented algebras.
\end{rmk}

\begin{cor}
    Let $A$ be an algebra and let $M$ be an $A$-bimodule with a right $A$-linear quasi-isomorphism $A\simeq M$. The following are equivalent:
    \begin{enumerate}
        \item $M$ is invertible.
        \item There exists an invertible $X$ such that $A\simeq M_X$ as bimodules.
        \item There exists an invertible $X$ such that $A\simeq {}_XM$ as bimodules.
        \item The left action $A\to \R\enn_A(M)$ is a quasi-isomorphism.
        \item There exists a homotopy automorphism $\phi$ of $A$ such that $A\simeq {}_\phi M$ as bimodules.
   \end{enumerate}
    \end{cor}
\begin{proof}
   The equivalence of (1), (2), and (3) is straightforward and does not use the assumption that $M\simeq A$. Clearly (5) implies (3). If $M$ is invertible and $N$ is any bimodule, then since $N\mapsto N_M$ is a dg autoequivalence, we have natural quasi-isomorphisms $\R\enn_A(N_M)\simeq \R\enn_A{N}$ and in particular $A\simeq \R\enn_A(M)$. Hence (1) implies (4). Finally, (4) implies (5) by \ref{nakesistence}.
\end{proof}

\section{Smoothness, properness, regularity \& locality}\label{spsect}
We develop the above concepts for algebras and coalgebras, then we compare them across Koszul duality. We deduce some properties of the `derived completion' functor $A\mapsto A^!\coloneqq (BA)^*$; in particular that $A\mapsto A^{!!}$ preserves smoothness and properness when $A^{!!}$ is viewed as a pseudocompact algebra. We also give partial converses.

 \subsection{Properness and locality for algebras}
In this section we study proper algebras. Recall that $A$ is proper if and only if $H^*(A)$ is a finite dimensional graded vector space; this is equivalent to $A\in \pvd(A)$. It is clear that $A$ is proper if and only if there is an inclusion $\per(A) \subseteq \pvd(A)$. A finite dimensional algebra is clearly proper. Not all proper algebras have finite dimensional models, due to an example of Efimov \cite[5.4]{efimovhodge}. However, connective algebras are better behaved:

\begin{thm}[{\cite[3.12]{rsproper}}; after {\cite[2.20]{orlovfd}}]\label{pcthm}
    If $A$ is a proper connective algebra then $A$ is quasi-isomorphic to a finite dimensional algebra.
\end{thm}

\begin{defn}
Let $A\to k$ be an augmented algebra. Say that $A$ is \textbf{homologically local} if $\pvd(A) = \thick_{D(A)}(k)$.
\end{defn}
We will often abbreviate `homologically local' to simply `local'. For discrete algebras coming from algebraic geometry, this definition is equivalent to the usual one: 
\begin{prop}
    Let $A$ be a noetherian discrete commutative $k$-algebra and $A\to k$ an augmentation. Then $A$ is a local ring if and only if $A$ is homologically local.
\end{prop}

\begin{proof}Let $M$ be a finite dimensional $A$-module, and so necessarily of finite length (i.e.\ both noetherian and Artinian; see \cite[\S2.4]{eisenbud} for more about length). Hence $M$ has a finite composition series with subquotients of the form $A/\mathfrak{m}_i$ with each $\mathfrak{m}_i$ maximal. Since a perfectly valued $A$-module is a finite iterated cone between finite dimensional modules, it follows that $\pvd(A)$ is contained in the subcategory $\thick_{D(A)}\{A/\mathfrak{m}:\mathfrak{m}\in \mathrm{MaxSpec}(A)\}$. If $A$ is local, then this latter subcategory is precisely $\thick(k)$, and hence $A$ is homologically local. Conversely, if $A$ is not local, let $\mathfrak{m}$ be the kernel of the augmentation $A\to k$ and choose a closed point $\mathfrak{p}\neq \mathfrak m$. Then $k(\mathfrak{p}) \in \pvd(A)$ by the Nullstellensatz. But by construction $k(\mathfrak{p})_\mathfrak{m}=0$ and hence $k(\mathfrak{p}) \notin \thick(k)$.\end{proof}

\p Say that a discrete finite dimensional algebra $A$ is \textbf{local} if the maximal semisimple quotient of $A$ is a copy of $k$. This is compatible with our earlier definition:
\begin{prop}\label{connloc}
    Let $A$ be a connective algebra such that $H^0A$ is finite dimensional. Then $A$ is homologically local if and only if $H^0A$ is local.
\end{prop}
\begin{proof}
For brevity put $B\coloneqq H^0A$. Since $A$ is connective, $D(A)$ admits a t-structure with heart $\Mod B$. First assume that $B$ is local. If $M$ is a finitely generated $B$-module, filtering $M$ by its radical filtration shows that $M$ can be written as an iterated extension between copies of $k$. In other words, we have a sequence $k=M_1,\ldots, M_m=M$ of $B$-modules and short exact sequences $0\to k\to M_{n+1} \to M_n\to 0$. These short exact sequences give exact triangles $k \to M_{n+1} \to M_n \to$ in $D(B)$ showing inductively that each $M_n$, and hence $M$, are in $\thick_{D(B)}(k)$. Since $A$ is connective, the natural $t$-structure on $D(A)$ yields a triangle functor $D(B) \to D(A)$ and hence $M \in \thick_{D(A)}(k)$. To finish, take $N\in \pvd(A)$; viewing $N$ as a finite iterated cone between finitely generated $B$-modules (namely the $H^iN$) we see that $N\in \thick_{D(A)}(k)$ as required. Conversely if $B$ is not local, write the maximal semisimple quotient of $B$ as $S\oplus k$ where $S\not\cong 0$. Then the $A$-module $S$ is an object of $\pvd(A)$ which is not in $\thick(k)$.
\end{proof}

It will also be convenient for us to introduce some slightly stronger notions of properness:
\begin{defn}
    Let $A$ be an augmented algebra. Say that $A$ is \textbf{proper local} if $A\in \thick_{D(A)}(k)$. Say that $A$ is \textbf{strongly proper local} if $A\in \thick_{D(A^e)}(k)$.
\end{defn}

If $A$ is an augmented algebra, then it is easy to see that $$A\text{ strongly proper local}\implies A \text{ proper local} \impliedby A \text{ proper and homologically local}$$
and if $A$ is connective then these implications can be reversed:

\begin{prop}\label{cSPisP}
    Let $A$ be a connective algebra. The following are equivalent:
\begin{enumerate}
    \item $A$ is proper local.
    \item $A$ is strongly proper local.
    \item $A$ is proper and homologically local.
\end{enumerate}
\end{prop}
\begin{proof}
If (1) holds then by \cite[5.7]{rsproper} combined with \ref{connloc} we see that (2) holds. Since (2) clearly implies (1) we see that (1) and (2) are equivalent. Moreover, certainly (3) implies (1) so we just need to show that (1) implies (3). To do this, certainly if (1) holds then $A$ is proper, so we just need to prove that it is local. By \ref{connloc} we only need to prove that the finite dimensional algebra $H^0A$ is local. But by assumption it is an iterated extension between copies of $k$, and hence must be local.
\end{proof}

\begin{rmk}
We are not aware of an example of a proper local nonconnective algebra which is not strongly proper local, but we expect Efimov's example \cite[5.4]{efimovhodge} to have this property.
\end{rmk}

We finish this section with a locality property that will be useful to us later.

If $\mathcal{T}$ is a triangulated category, recall that a thick subcategory $\mathcal{C} \into \mathcal{T}$ is said to be \textbf{localising} if it is closed under infinite direct sums and \textbf{colocalising} if it is closed under infinite products. If $\mathcal{X}$ is a set of objects of $\mathcal{T}$, recall that $\mathrm{Loc}_{\mathcal{T}}(\mathcal{X})$ denotes the smallest localising subcategory containing $\mathcal{X}$, and dually $\mathrm{Coloc}_{\mathcal{T}}(\mathcal{X})$ denotes the smallest colocalising subcategory containing $\mathcal{X}$. The objects of $\mathcal{X}$ are generators for $\mathrm{Loc}_{\mathcal{T}}(\mathcal{X})$ and cogenerators for $\mathrm{Coloc}_{\mathcal{T}}(\mathcal{X})$.

\begin{defn}
    Say that an augmented algebra $A$ is \textbf{homologically complete local} if $A\in \mathrm{Coloc}_{D(A)}(k)$.
\end{defn} As before we will frequently abbreviate `homologically complete local' to `complete local'. Clearly a proper local algebra is complete local. 

\begin{rmk}
    If $(A,\mathfrak{m},k)$ is a commutative noetherian regular local ring then $\mathrm{Loc}(k)$ is the derived category of $\mathfrak{m}$-torsion modules, and $\mathrm{Coloc}(k)$ is the derived category of derived $\mathfrak{m}$-complete modules.
\end{rmk}

\subsection{Smoothness and regularity for algebras}

\begin{defn}
An algebra $A$ is \textbf{homologically smooth} or just \textbf{smooth} if $A$ is a perfect $A$-bimodule.
\end{defn}

\begin{rmk}
If $A$ is a commutative finite type discrete $k$-algebra, then $A$ is homologically smooth if and only if the unit morphism $k \to A$ is smooth. When $k$ is a perfect field, these are in turn equivalent to $A$ being regular, which is also equivalent to $A$ having finite global dimension.
\end{rmk}

\begin{rmk}
If $A$ is a discrete algebra, then $\gldim A^e\geq 2\gldim A$. If $k$ is a perfect field and $A$ is finite dimensional, then this inequality becomes an equality \cite[Theorem 16]{auslanderglobal}; in particular if $A$ has finite global dimension then $A$ is homologically smooth.
\end{rmk}

Following \cite{KSreflex}, we say that an algebra $A$ is \textbf{hfd-closed} if there is an inclusion $\pvd(A) \subseteq \per A$. This is often a convenient replacement for smoothness, as the following proposition shows:

\begin{prop}\label{fdmodel}
Let $A$ be an algebra.
\begin{enumerate}
    \item If $A$ is smooth then $A$ is hfd-closed.
    \item If $A$ is (quasi-isomorphic to) a finite dimensional algebra, then $A$ is smooth if and only if it is hfd-closed.
\end{enumerate}
\end{prop}
\begin{proof}
Claim $(1)$ is standard: if $M\in \pvd A$ and $P$ is a perfect $A$-bimodule then then $M\lot_A P \in \per A$; it is enough to show this for $P=A^e$ which is clear. In particular, if $A$ is smooth then we may take $P=A$. Claim (2) follows from a theorem of Orlov \cite[1.11]{orlovsmooth}.
\end{proof}

\begin{rmk}
    If $A$ is an augmented algebra with Koszul dual coalgebra $C=BA$, then $A$ is hfd-closed if and only if there is an inclusion $\hqfco C \subseteq \thick_{\dco(C)}(k)$.
\end{rmk}

\begin{defn}
    A morphism of algebras $A\to l$ is \textbf{regular} if it makes $l$ into a perfect $A$-module. We say that an augmented algebra is \textbf{regular} if the augmentation morphism is regular. This notion is called \textbf{g-regular} in \cite{gsmorita}.
\end{defn}

Smoothness and regularity are closely related:
\begin{lem}\label{smreglem}
    Let $A$ be an augmented algebra.
    \begin{enumerate}
        \item If $A$ is regular local then it is hfd-closed.
        \item If $A$ is hfd-closed then it is regular.
        \item If $A$ is smooth then it is regular.
    \end{enumerate}
\end{lem}
\begin{proof}
    Follows easily from the definitions.
\end{proof}

\begin{lem}Let $A$ be a discrete commutative finite type $k$-algebra and $\mathfrak{p}$ a prime ideal of $A$, with residue field $l$. If $A_\mathfrak{p}$ is smooth then $A_\mathfrak{p}\to l$ is regular. If $A_\mathfrak{p}$ is regular and $l$ is a separable extension of $k$ then $A_\mathfrak{p}$ is smooth.
\end{lem}
\begin{proof}
    It is standard that $A_\mathfrak{p}\to l$ is regular exactly when $A_\mathfrak{p}$ is a regular local ring. Since $A$ was finite type, it is smooth at $\mathfrak{p}$ if and only if it is geometrically regular at $\mathfrak{p}$, and the result follows.
\end{proof}
\begin{cor}
 Let $A\to k$ be a discrete commutative augmented finite type $k$-algebra with augmentation ideal $\mathfrak{m}$. Then the local ring $A_\mathfrak{m}$ is smooth if and only if $A_\mathfrak{m} \to k$ is regular.   
\end{cor}

\begin{prop}
    Let $A$ be quasi-isomorphic to a finite dimensional augmented algebra, and suppose in addition that $A$ is local. Then $A$ is regular if and only if is is smooth.
\end{prop}
\begin{proof}
If $A$ is smooth then it is regular. Conversely, if $A$ is regular then assume without loss of generality that $A$ is finite dimensional. Then it is smooth by {\cite[3.12]{orlovsmooth}}.
\end{proof}

\subsection{Properness for coalgebras}

The notion of properness for coalgebras is more subtle than that for algebras. Indeed, properness was earlier defined in terms of the isomorphism-reflecting forgetful functor $D(A)\to D(k)$, whereas the forgetful functor $\dco(C) \to D(k)$ does not reflect isomorphisms. However, it is possible to identify a large subcategory of $\dco(C)$ where this does hold:

\begin{prop}\label{qisoreflec}
Let $C$ be a coalgebra. The natural map $\mathrm{Coloc}_{\dco(C)}(C) \to D(k)$ reflects isomorphisms. 
\end{prop}
\begin{proof}
Let $M \in \mathrm{Coloc}_{\dco(C)}(C)$ be an acyclic comodule. We wish to show that $M$ is coacyclic. Following \cite[5.5]{positselskitwokinds}, we see that the full subcategory of $\dco(C)$ consisting of the acyclic $C$-comodules is precisely the left orthogonal to $C$. But this is also the left orthogonal to $\mathrm{Coloc}_{\dco(C)}(C)$. So $M$ is simultaneously an object of both $\mathrm{Coloc}_{\dco(C)}(C)$ and its left orthogonal, and hence is zero in $\dco(C)$, as required.
\end{proof}
\begin{rmk}
    There is a natural localisation functor $\dco(C) \to D(\Comod C)$ whose kernel consists of the acyclic comodules. Under the assumption of Vop\v{e}nka's principle, there is a natural equivalence
$\mathrm{Coloc}_{\dco(C)}(C) \simeq D(\Comod C)$ so that $\mathrm{Coloc}_{\dco(C)}(C)$ is actually the largest subcategory of $\dco(C)$ for which the natural functor to $D(k)$ reflects isomorphisms. See \cite[5.5, Note added three years later]{positselskitwokinds} for further discussion.
\end{rmk}

\begin{lem}
    Let $(A,C)$ be a Koszul duality pair. Then $A$ is complete local if and only if $\fd{C} \subseteq \mathrm{Coloc}_{\dco(C)}(C)$.
\end{lem}
\begin{proof}
    Since colocalising subcategories are thick, $A$ is complete local if and only if $\per(A)\subseteq \mathrm{Coloc}_{D(A)}(k)$. Across Koszul duality this corresponds to the desired statement.
\end{proof}

\begin{defn}
    Let $C$ be a coalgebra. Say that $C$ is \textbf{proper} if $C\in \fd(C)$. Say that $C$ is \textbf{strongly proper} if $C\in \fd(C^e)$.
\end{defn}
A finite dimensional $C$-bicomodule is certainly a finite dimensional $C$-module, so a strongly proper coalgebra is proper. If $C$ is a proper coalgebra then certainly $C\in \pvd(k)$.

\begin{rmk}
    In \cite{HRCY}, different terminology is used: their `proper' corresponds to our `strongly proper'.
\end{rmk}

\begin{prop}\label{kdpropc}
    Let $(C,A)$ be a Koszul duality pair.
    \begin{enumerate}
        \item $C$ is proper if and only if $A$ is regular.
        \item $C$ is strongly proper if and only if $A$ is smooth.
        \item If $A$ is complete local, the following are equivalent:
      \begin{enumerate}
        \item $C$ is proper.
        \item $C$ is quasi-isomorphic to a finite dimensional $C$-comodule.
    \end{enumerate}
    \end{enumerate}
   
\end{prop}
\begin{proof}
    For the first claim, we see that the condition $C\in \fd(C)$ translates across Koszul duality to $k \in \per(A)$, which is the definition of regularity. For the second claim, the condition $C\in \fd(C^e)$ translates across bimodule Koszul duality to $A\in \per(A^e)$. For the third claim, clearly (a) implies (b). If (b) holds, choose a finite dimensional comodule $M$ with a quasi-isomorphism $C\simeq M$. Since $A$ was complete local, $M$ is an object of $\mathrm{Coloc}_{\dco(C)}(C)$, and hence the quasi-isomorphism $C\simeq M$ was actually a weak equivalence, so (a) holds.
\end{proof}
\begin{rmk}
    If $C$ is a coconnective coalgebra such that $C\in \pvd(k)$, then $C$ is quasi-isomorphic to a finite dimensional $C$-comodule: indeed the algebra $C^*$ is connective and proper, and hence quasi-isomorphic to a finite dimensional algebra (by \ref{pcthm}), and in particular a module over itself.
\end{rmk}

\subsection{Smoothness for coalgebras}

\begin{defn}
    Let $C$ be a coalgebra. Say that $C$ is \textbf{smooth} if $C\in \thick_{\dco(C^e)}(C^e)$. Say that $C$ is \textbf{regular} if $k\in \thick_{\dco(C)}(C)$.
\end{defn}

\begin{prop}\label{cogsmooth}
    Let $(C,A)$ be a Koszul duality pair. Then $C$ is smooth if and only if $A$ is strongly proper local. $C$ is regular if and only if $A$ is proper local. If $C$ is smooth then it is regular. 
\end{prop}
\begin{proof}
    The first two claims are a simple translation across Koszul duality. The third claim is clear: if $C$ is weakly equivalent as a bicomodule to a finite dimensional bicomodule, then $C$ is certainly weakly equivalent as a right comodule to a finite dimensional right comodule.
\end{proof}

\begin{thm}\label{smpropdualcon}
    Let $(C,A)$ be a Koszul duality pair, with $A$ a connective local algebra. Then $A$ is proper if and only if $C$ is smooth.
\end{thm}
\begin{proof}
If $C$ is smooth then $A$ is certainly proper. Conversely, if $A$ is proper then it is strongly proper local by \ref{cSPisP}, and hence $C$ is smooth by \ref{cogsmooth}.
\end{proof}

\begin{lem}
    Let $C$ be a coalgebra. Then:
    \begin{enumerate}
        \item If $C$ is strongly proper then $\hqfco(C)\subseteq \fd(C)$.
        \item If $\hqfco(C)\subseteq \fd(C)$ then $C$ is proper.
        \item If $\Omega C$ is local then $\hqfco(C)\subseteq \fd(C)$ if and only if $C$ is proper.
        \item If $C$ is smooth and $\Omega C$ is local and connective then $C$ is strongly proper if and only if $\hqfco(C)\subseteq \fd(C)$.
    \end{enumerate}

\end{lem}
\begin{proof}
    Claim (1) is the dual of \ref{fdmodel}(1), and claim (2) is the dual of \ref{smreglem}(2). Claim (3) is the dual of \ref{smreglem}(1). For claim (4), suppose that $\hqfco(C)\subseteq \fd(C)$. Putting $A\coloneqq \Omega C$ we see that $A$ is connective and proper and hence quasi-isomorphic to a finite dimensional algebra. By \ref{fdmodel}(2), $A$ is smooth, and the claim follows by \ref{smpropdualcon}.
\end{proof}

\subsection{Derived completions}
If $A$ is an augmented algebra, we write $A^!\coloneqq (BA)^*$, which is a pseudocompact algebra. Forgetting the pseudocompact topology, we may iterate this construction and consider the pseudocompact algebra $A^{!!}$, which we call the \textbf{derived completion} of $A$, following \cite{efimovcmp}. There is a natural map $A \to A^{!!}$ which in general is not a quasi-isomorphism; for example for $A=k[t]$ it is the $(t)$-completion $k\llbracket t \rrbracket$. However, we may ask when $A^{!!}$ inherits properties of $A$; in this section we show that it inherits smoothness (as a pseudocompact algebra) and properness; we also give converse theorems.

\begin{prop}\label{kddsmooth}
Let $A$ be an augmented algebra.
\begin{enumerate}
    \item If $A$ is smooth then so is $BA^!$.
    \item If $BA^!$ is smooth and $A^e$ is complete local then $A$ is smooth.
\end{enumerate}
\end{prop}
\begin{proof}
Put $C\coloneqq BA$, so that $A^!=C^*$. If $A$ is smooth then $C$ is strongly proper and so $C \in \fd(C^e) \simeq \thick_{\dco(C^e)}(k)$. Linearly dualising this, and using that $(C^e)^* \simeq (C^e)^*$ since $C \in \pvd(k)$, we see that $C^*$ is strongly proper. Hence $BC^*=BA^{!}$ is smooth, proving (1). On the other hand, if $BA^!$ is smooth then $C^* \in \thick_{D(C^{*e})}(k) \simeq \thick_{D(C^{e*})}(k)$. Bimodule Koszul duality yields an equivalence $\mathrm{Coloc}_{\dco(C^e)}(C^e)\simeq \mathrm{Coloc}_{D(A^e)}(k)$, and hence if $A^e$ is complete local then $k \in \mathrm{Coloc}_{\dco(C^e)}(C^e)$. In particular, the linear dual gives a contravariant equivalence $\thick_{D(C^{e*})}(k) \simeq \thick_{\dco(C^e)}(k)$. It now follows that $C \in \thick_{\dco(C^e)}(k)$ and hence $A$ is smooth.
\end{proof}

\begin{rmk}
    One only expects $A^e$ to be complete local when $A$ is proper local; for example the enveloping algebra of $k\llbracket t\rrbracket$ is not complete local. Instead, the correct object to work with is the \textit{completed} enveloping algebra, which is the linear dual of the coalgebra $k\llbracket t\rrbracket^{*e}$.
\end{rmk}

\begin{prop}\label{kddproper}
    Let $A$ be a local algebra.
    \begin{enumerate}
        \item If $A$ is proper then so is $A^{!!}$.
        \item If $A^{!!}$ is local and quasi-isomorphic to a finite dimensional algebra then $A$ is proper.
    \end{enumerate}
\end{prop}
Before we begin the proof, we isolate a useful lemma.
\begin{lem}\label{pvdcog}
If $C$ is a coalgebra, then linear duality gives a quasi-equivalence $$\thick_{\dco(C)}(C)\simeq \per(C^*)^\circ.$$
\end{lem}
\begin{proof}
We have $\R\enn_{\dco(C)}(C)\simeq C^*$ since $C$ is an injective $C$-comodule. Hence we obtain an abstract quasi-equivalence $\thick_{\dco(C)}(C)\simeq \per(C^*)^\circ$ of pretriangulated dg categories that sends $C$ to $C^*$. Since the linear dual is a triangle functor, and agrees with this quasi-equivalence on the generator $C$, it must be a quasi-equivalence.
\end{proof}

\begin{proof}[Proof of \ref{kddproper}]
Put $C=BA$. Observe that:
\begin{align*}
    A \text{ is proper }&\iff C \text{ is regular}\\
    &\iff k\in \thick_{\dco(C)}(C)\\
    &\iff k \in \per(C^*) \text{ (by \ref{pvdcog})}\\
    &\iff A^! \text{ is regular}\\
    &\iff BA^{!} \text{ is proper.}
\end{align*}
Now it is clear that (1) holds, since if $BA^!$ is proper then its linear dual $A^{!!}$ is certainly proper. For (2), if $A^{!!}$ has a finite dimensional model then since $A^{!!}$ is local we see that $A^{!!} \in \thick_{D(A^{!!})}(k)$. Linearly dualising, we see that $BA^!$ is proper and hence $A$ is proper.
\end{proof}

\begin{rmk}\label{moritaduality}
If $\mathcal{A}$ is a dg category, there is a natural `evaluation' map $\mathrm{ev}_\mathcal{A}:\mathcal{A} \to \pvd \pvd \mathcal{A}$. Following \cite{KSreflex, Greflex} say that $\mathcal{A}$ is \textbf{Morita reflexive} if $\mathrm{ev}_\mathcal{A}$ is a Morita equivalence. The terminology comes from the fact that $$\pvd \mathcal{A} \simeq \R\hom_{\mathrm{Hqe}}(\mathcal{A},\per k)\simeq \R\hom_{\mathrm{Hmo}}(\mathcal{A},k)$$ is the dual of $\mathcal{A}$ in the Morita homotopy category. If $A$ is a local algebra, then by \ref{pvdcog} there is a Morita equivalence $\pvd(A) \simeq (A^!)^\circ$. In particular, if $A^!$ is also local then the natural map $\mathrm{ev}_A$ agrees up to Morita equivalence with the derived completion map $c:A \to A^{!!}$, and so $A$ is Morita reflexive if and only if $c$ is a quasi-isomorphism. A large class of derived complete in this sense algebras is identified in \cite{ddefspt}. We also remark that there is a Koszul dual notion of Morita reflexivity for coalgebras, defined in terms of the natural map $\fd C\to\pvd\hqfco C$. The first author will further explore the connections between Morita reflexivity and Koszul duality in upcoming work joint with Isambard Goodbody and Sebastian Opper.
\end{rmk}

\section{Calabi--Yau conditions}
We give definitions of what it means for an algebra to be (twisted, nonsmooth) Calabi-Yau. We show that a twisted nonsmooth CY algebra admits a twisted CY structure on certain subcategories of its derived category.

\subsection{The dualising complex}
Let $A$ be an algebra and $M$ an $A$-bimodule. Write $\HH^\bullet(A,M)\coloneqq \R\hom_{A^e}(A,M)$ for the Hochschild cohomology complex and $\HH_\bullet(A,M)\coloneqq M \lot_{A^e}A$ for the Hochschild homology complex. When $M=A$ we will usually drop it from the notation; beware that $A\mapsto \HH^\bullet (A)$ is not a functor. However, the assignment $A\mapsto \HH^\bullet(A,A^*)$ is a contravariant functor from algebras to vector spaces; in the older literature this functor is also known as the Hochschild cohomology of $A$. Both $\HH^\bullet(A,M)$ and $\HH_\bullet(A,M)$ are vector spaces; if $M$ is not just an $A$-bimodule, but an $R$-$A^e$-bimodule for some algebra $R$, then they are $R$-modules. In particular when $A$ is commutative then $M$ is a (symmetric) $A^e$-bimodule, and hence both $\HH^\bullet(A,M)$ and $\HH_\bullet(A,M)$ are $A$-bimodules.

\begin{defn}
Let $A$ be an algebra. The \textbf{inverse dualising complex} of $A$ is the $A$-bimodule $A^\vee \coloneqq \R\hom_{A^e}(A,A^e)\simeq \HH^\bullet(A,A^e)$. 
\end{defn}


\p If $M$ is a perfect bimodule and $N$ is any bimodule, then there is a natural quasi-isomorphism $\R\hom_{A^e}(M,N)\coloneqq N\lot_{A^e}M^\vee$; to see this write $M$ as an iterated cone between copies of $A^e$. In particular if $A$ is homologically smooth, then we get natural quasi-isomorphisms of vector spaces $$\HH^\bullet(A,M)\simeq M \lot_{A^e}A^\vee$$
$$\HH_\bullet(A,M)\simeq M\lot_{A^e}A^{\vee\vee}\simeq\R\hom_{A^e}(A^\vee,M)$$ where in the second line we use that $P^{\vee\vee}\simeq P$ for a perfect $A$-bimodule $P$. When $A$ is commutative and $M$ is a symmetric bimodule then these are quasi-isomorphisms of $A$-bimodules.

\begin{rmk}
Suppose that $A$ is a homologically smooth commutative algebra. Substituting $M=A$ into the first quasi-isomorphism above gives a quasi-isomorphism $ \HH^\bullet(A)\simeq \HH_\bullet(A,A^\vee)$. Since $A$ is commutative, both sides are $A$-bimodules, and this quasi-isomorphism is $A$-bilinear. So dualising we obtain an $A$-bilinear quasi-isomorphism $ \HH^\bullet(A)^\vee\simeq \HH_\bullet(A,A^\vee)^\vee$. Using the hom-tensor adjunction, the right hand side simplifies to
 $\HH^\bullet(A)$, and we obtain an $A$-bilinear quasi-isomorphism $\HH^\bullet(A)^\vee\simeq \HH^\bullet(A)$. In other words, the bimodule $\HH^\bullet(A)$ is self-dual. 
\end{rmk}

\subsection{Calabi--Yau algebras}

\begin{defn}
    Let $A$ be an algebra. Say that $A$ is \textbf{twisted nonsmooth Calabi--Yau} if $A^\vee$ is an invertible $A$-bimodule. Say that $A$ is \textbf{nonsmooth $n$-Calabi--Yau} if
    there is an $A$-bilinear quasi-isomorphism $A\simeq A^\vee[n]$. Say that $A$ is \textbf{(twisted) Ginzburg Calabi--Yau} if it is both smooth and (twisted) nonsmooth Calabi--Yau.
\end{defn}

Note that a nonsmooth $n$-Calabi--Yau algebra is twisted nonsmooth Calabi--Yau, since $A[-n]$ is invertible.

\begin{rmk}
    This definition is a noncommutative-geometric abstraction of the notion of `trivial canonical bundle' in algebraic geometry; indeed a Calabi--Yau variety is precisely a smooth variety with trivial canonical bundle.
\end{rmk}
Since nonsmooth Calabi--Yau algebras are our main focus, we will frequently abbreviate ``nonsmooth Calabi--Yau'' as simply \textbf{CY}, with the understanding that \textbf{Ginzburg CY} refers to Calabi--Yau algebras in the usual sense. We will also often suppress the integer $n$ from the notation.

\begin{rmk}
    The original definition of Ginzburg CY in \cite{ginzburg} requires in addition that the isomorphism $f:A\to A^\vee[n]$ in the derived category of $A$-bimodules satisfies the extra condition that the composition $A \to A^{\vee\vee} \xrightarrow{f^\vee[n]} A^\vee[n]$ is equal to $f$. However, it is a theorem of Van den Bergh that in the smooth setting, this extra condition is superfluous \cite[C.1]{vdb}. In the nonsmooth setting, this extra condition does not seem useful.
\end{rmk}

\begin{rmk}
If $A$ is a smooth algebra, then an $A$-bilinear quasi-isomorphism $A\simeq A^\vee[n]$ is equivalent to the data of a nondegenerate degree $n$ class $\eta$ in the Hochschild complex $HH_\bullet(A)$. When $\eta$ lifts to a class in negative cyclic homology, one says that $A$ is a \textbf{left Calabi--Yau} or \textbf{exact Calabi--Yau} algebra.  In \cite{HRCY}, exact CY algebras are referred to as \textbf{smooth Calabi--Yau algebras}, with Ginzburg CY algebras referred to as \textbf{weak smooth Calabi--Yau algebras}. There is a dual version for proper algebras known as \textbf{right Calabi--Yau}. When $A$ is neither smooth nor proper, there seems to be no interpretation of a nonsmooth CY structure in terms of a Hochschild class, and hence no notion of `exact nonsmooth CY' algebras.
\end{rmk}

\begin{rmk}
    Suppose that $A$ is a smooth algebra with a right $A$-linear quasi-isomorphism $A\simeq A^\vee[n]$. Then $A$ is twisted Ginzburg Calabi--Yau: since $A$ is smooth, the natural map $A\to\R\enn_A(A^\vee)\simeq A^{\vee\vee}$ is a quasi-isomorphism, and an application of \ref{nakesistence} shows that there exists a homotopy automorphism $\nu$ of $A$ and a two-sided quasi-isomorphism $A\simeq {}_\nu A^\vee [n]$.
\end{rmk}

\begin{rmk}\label{fracCYrmk}
    One can also make sense of a fractional version of the above definition: say that $A$ is fractional CY if there exists some $m>0 $ such that $(A^\vee)^{\lot_A m}\simeq A$. Fractional CY algebras abound in the literature; for a selection of examples see \cite{fracHI, fracFK, fracK, fracR, fracG}.
\end{rmk}

\begin{lem}\label{prodlem}
Let $A$ be a proper twisted CY algebra and let $B$ be a proper twisted CY algebra. Then the tensor product $A\otimes B$ is a proper twisted CY algebra.
\end{lem}
\begin{proof}
Since both $A$ and $B$ are proper, a standard computation with the bar complex shows that we have $\HH^\bullet(A,M)\otimes \HH^\bullet(B,N) \simeq \HH^\bullet(A\otimes B, M\otimes N)$ whenever $M$ is a proper $A$-bimodule and $N$ is a proper $B$-bimodule. In particular, we have $(A\otimes B)^\vee\simeq A^\vee\otimes B^\vee$ and the result follows.
\end{proof}
In particular, the proof shows that if $A$ is a proper $n$-CY algebra and $B$ is a proper $m$-CY algebra then $A\otimes B$ is a proper $(n+m)$-CY algebra.

\subsection{Calabi--Yau conditions on categories}

\begin{defn}
    Let $A$ be an augmented algebra and $\mathcal{L},\mathcal{M}$ two triangulated subcategories of $D(A)$. We say that the pair $(\mathcal{L},\mathcal{M})$ has a \textbf{twisted Calabi--Yau structure} if there is an invertible bimodule $X$ such that for all $L\in \mathcal{L}$ and all $M\in \mathcal{M}$, there are natural quasi-isomorphisms $$\R\hom_A(M,L_X)^* \simeq \R\hom_A(L, M).$$When $X=A[n]$ we call this an \textbf{$n$-Calabi--Yau structure}. When $\mathcal{L}=\mathcal{M}$ we simply say that $\mathcal{L}$ has a \textbf{twisted Calabi--Yau structure}.
\end{defn}
Observe that in the above definition, the functor $S(L) \coloneqq L_X$ is a relative Serre functor, and indeed this condition is equivalent to the existence of a twisted CY structure. 

\p If $A$ is Ginzburg $n$-CY then the pair $(D(A),\pvd(A))$ has a $n$-CY structure \cite[3.4]{kellerCYC}. If $A$ is smooth, proper, and finitely presented, then $[n]$ is a Serre functor on $\pvd(A)$ if and only if $A$ is Ginzburg $n$CY \cite[3.2.4]{ginzburg}.

\p The following lemma will allow us to produce CY structures on the derived categories of CY algebras.

\begin{lem}\label{kellerlem}
   Let $A$ be an algebra. Let $L\in D(A)$ and let $M\in \pvd(A)$. There are natural maps
   $$L\lot_A A^\vee \lot_A M^* \to \R\hom_A(M,L)$$
   $$\R\hom_A(M,L)^* \to \R\hom_A(L\lot_A A^\vee, M)$$which are quasi-isomorphisms if $P\coloneqq \R\hom_k(M,L)$ is a perfect $A$-bimodule.
\end{lem}
\begin{proof}

    We follow the proof of \cite[4.1]{kellerCYC}. There are natural maps \begin{align*}
        L\lot_A A^\vee \lot_A M^* &\simeq P\lot_{A^e} A^\vee \xrightarrow{\theta} \R\hom_{A^e}(A,P) \simeq \R\hom_A(M,L)
    \end{align*}and the first map of interest is their composition. Taking its linear dual, we obtain a natural map
    \begin{align*}
        \R\hom_A(M,L)^* \to (L\lot_A A^\vee \lot_A M^*)^* \simeq \R\hom_A(L\lot_A A^\vee, M)
    \end{align*}which is our second desired map. Observe that the two natural maps are quasi-isomorphisms if and only if $\theta$ is. To prove that $\theta$ is a quasi-isomorphism if $P$ is perfect, it suffices to prove it for $P=A^e$, which is clear.
\end{proof}

\begin{cor}\label{kellercor}
 Let $A$ be an algebra, let $L\in \per A$ and let $M\in \pvd(A)$. Suppose that either:
 \begin{enumerate}
     \item $A$ is smooth
     \item $A$ is augmented, $A^\circ$ is regular, and $M\in \thick_{D(A)}(k)$.
 \end{enumerate}
Then there is a natural quasi-isomorphism $$\R\hom_A(M,L)^* \to \R\hom_A(L\lot_A A^\vee, M).$$
\end{cor}
\begin{proof}
By \ref{kellerlem} it suffices to show that $\R\hom_k(M,L)\simeq M^*\otimes L$ is a perfect $A$-bimodule. Since $L$ is perfect by assumption, it suffices to show that $M^* \in \per(A^\circ)$. If (1) holds then $A^\circ$ is smooth and hence $M^*\in\pvd(A^\circ) \subseteq \per(A^\circ)$. If (2) holds then $M^* \in \thick_{D(A^\circ)}(k) \subseteq \per(A^\circ)$.
\end{proof}

The following result (or at least its untwisted version) is well known and appears in e.g. \cite{ginzburg, kellerCYC}.

\begin{thm}\label{openGCY}
    Let $A$ be a twisted Ginzburg CY algebra. Then the pair $(\per A, \pvd A)$ has a twisted CY structure.
\end{thm}
\begin{proof}
Let $L\in \per(A)$ and $M\in \pvd(A)$. Let $X$ be the inverse of the invertible bimodule $A^\vee$. Then by \ref{kellercor}(1) along with the twisted CY property, we have natural quasi-isomorphisms
\begin{align*}
\R\hom_A(M,L_X)^* &\simeq \R\hom_A(L_X\lot_A A^\vee, M) \\
& \simeq \R\hom_A(L, M) 
\end{align*}as required.
\end{proof}
One can prove a very similar result when $A$ is not smooth, but only regular:
\begin{thm}\label{openCY}
    Let $A$ be an augmented twisted CY algebra such that $A^\circ$ is regular. Then the pair $(\per A, \thick(k))$ has a twisted CY structure.
\end{thm}
\begin{proof}
    Identical to the proof of \ref{openGCY} except that one uses \ref{kellercor}(2).
\end{proof}

\section{Gorenstein and Frobenius algebras}
\subsection{Gorenstein morphisms}
Let $g:A \to l$ be a map of algebras. Note that the one-sided dual $\R\hom_A(l,A)$ of $l$ is naturally an $A$-$l$-bimodule, and in particular an $A$-bimodule. Say that $g$ is \textbf{one-sided $d$-Gorenstein} if there is a quasi-isomorphism $l\simeq \R\hom_A(l,A)[d]$ of left $A$-modules. Say that $g$ is \textbf{two-sided $d$-Gorenstein} if there is a quasi-isomorphism $l\simeq \R\hom_A(l,A)[d]$ of $A$-bimodules. We will often drop the integer $d$ from the notation. 

\p We say that an augmented algebra $A$ is \textbf{$d$-Gorenstein} if the augmentation morphism $A\to k$ is two-sided $d$-Gorenstein. This definition is originally due to Avramov and Foxby \cite{afgor}; a spectral version was also given in \cite{DGI}. Frankild and {J\o}rgensen \cite{fjgorenstein} gave a `global' version of this definition where $k$ is replaced with $\pvd(A)$.

\p This condition is equivalent to the a priori weaker condition of $A\to k$ being one-sided Gorenstein:
\begin{prop}Let $g:A \to l$ be a map of algebras. \begin{enumerate}
\item $g$ is two-sided $d$-Gorenstein if and only if there is a quasi-isomorphism $l\simeq \R\hom_A(l,A)[d]$ of $A$-$l$-bimodules.
\item If $l=k$ (i.e.\ $g$ is an augmentation) then $g$ is one-sided $d$-Gorenstein if and only if it is two-sided $d$-Gorenstein.
\item If $A$ is commutative and $l$ is a symmetric bimodule, then $g$ is one-sided $d$-Gorenstein if and only if it is two-sided $d$-Gorenstein.
\end{enumerate}
\end{prop}
\begin{proof}
For brevity put $X\coloneqq \R\hom_A(l,A)$. To prove (1), consider the restriction functor $g^*$ from the derived category of $A$-$l$-bimodules to the derived category of $A$-bimodules. A quasi-isomorphism $l\simeq X[d]$ of $A$-$l$-bimodules hence yields a quasi-isomorphism $g^*l\simeq g^*X[d]$ of $A$-bimodules. But since the right action of $A$ on both $X$ and $l$ factors through the right action of $l$, we see that $g^*l\simeq l$ and $g^*X\simeq X$, and the claim follows. Claim (2) follows from (1) with $l=k$. For claim (3), observe that both $X$ and $l$ are symmetric $A$-bimodules, and hence they are quasi-isomorphic as bimodules if and only if they are quasi-isomorphic as left modules.
\end{proof}

\begin{rmk}
In the presence of some connectivity conditions, one can check the Gorenstein property on cohomology. Let $A$ be an augmented algebra and put $H\coloneqq H^*(A)$, which is a graded algebra with no differential. As usual if $M,N$ are two graded $H$-modules then we let $\ext^{p,q}_H(M,N)$ denote the $q^\text{th}$ graded piece of $\ext^{p}_H(M,N)$. There is a spectral sequence with $E_2$ page $$\ext^{p,q}_{H}(k,H) \implies \ext^{p+q}_A(k,A)$$ which converges if $A$ is connective enough. If $H$ is Gorenstein, then this spectral sequence collapses at the $E_2$ page to show that $A$ is Gorenstein.
\end{rmk}

Under some regularity assumptions, Calabi--Yau algebras are Gorenstein:

\begin{prop}\label{regCYisGor}
Let $A$ be an augmented $n$-CY algebra such that $A^\circ$ is regular. Then $A$ is $n$-Gorenstein.
\end{prop}
\begin{proof}
The CY structure on $(\per A, \thick(k))$ from \ref{openCY} yields a natural, and hence $A$-bilinear, quasi-isomorphism $\R\hom_A(k,A[n])^* \simeq k$. Taking the linear dual, we get an $A$-bilinear quasi-isomorphism $\R\hom_A(k,A)[n]\simeq k$, as desired.
\end{proof}

\begin{rmk}
    Let $A$ be a regular $d$-Gorenstein algebra and $S\subseteq H(A)$ a homogeneous susbset. Let $A\to L_SA$ be the derived localisation in the sense of \cite{bclloc} and assume that the augmentation module $k$ of $A$ is $S$-local. Then $L_SA$ is also Gorenstein: indeed, we have natural quasi-isomorphisms
    \begin{align*}
    \R\hom_{L_SA}(k,L_SA) &\simeq \R\hom_A(k,L_SA) & \text{ since }k\text{ is }S\text{-local}\\
    &\simeq \R\hom_A(k,A)\lot_A L_SA & \text{ since }A\text{ is regular}\\
    &\simeq k[-d]\lot_A L_SA & \text{ since }A\text{ is Gorenstein}
\end{align*}and since the localisation is smashing, $k[-d]\lot_A L_SA$ is the localisation of the $A$-module $k[-d]$, which by hypothesis is $k[-d]$. Hence $L_SA$ is $d$-Gorenstein.
\end{rmk}

\subsection{Discrete Gorenstein algebras}

In this part all algebras will be discrete. Say that a noncommutative ring is \textbf{Iwanaga--Gorenstein} if it is two-sided noetherian, and has finite injective dimension over itself, as both a left and right module. This is the usual definition of `Gorenstein' in the noncommutative algebra literature, but generalises less well to the dg setting. The following proposition is standard:

\begin{prop}[{\cite[18.1]{matsumura}}]
Let $A$ be a commutative noetherian local ring of finite Krull dimension with residue field $l$. The following are equivalent:
\begin{enumerate}
    \item $A$ is Iwanaga--Gorenstein.
    \item $A\to l$ is Gorenstein.
    \item $A\to l$ is $\dim(A)$-Gorenstein.
    \item $\ext^i_A(l,A)$ vanishes for $i\gg 0$.
\end{enumerate}
\end{prop}

Suppose that $A$ is a commutative noetherian ring. If $A$ is Iwanaga--Gorenstein, then so are all of its localisations at prime ideals, since injective resolutions localise. The converse is true if $A$ has finite Krull dimension \cite{ubiquity}. In particular we have:

\begin{prop}
Let $A$ be a commutative Iwanaga-Gorenstein ring. Every surjection $A\to l$ to a field is Gorenstein. 
\end{prop}
\begin{proof}
Let $\mathfrak {m}$ be the kernel of $A \to l$. It is easy to see that $\R\hom_A(l,A)$ is supported at $\mathfrak {m}$, where we have $\R\hom_A(l,A)_\mathfrak{m}[d]\simeq l$ since $A_\mathfrak{m}$ is (Iwanaga)--Gorenstein.
\end{proof}

Local complete intersection rings are Iwanaga--Gorenstein, and in particular any hypersurface ring of the form $k[x_1,\ldots,x_n]/f$ is Iwanaga--Gorenstein. In particular, polynomial algebras $k[x_1,\ldots,x_n]$, complete local polynomial algebras $k\llbracket x_1,\ldots,x_n\rrbracket$, and truncated polynomial algebras $k[y]/y^n$ are all Gorenstein when equipped with the standard augmentation.

\p Recall that a commutative $k$-algebra $A$ is \textbf{essentially of finite type} if it is a localisation of a finitely generated $k$-algebra. For example, the coordinate ring of a quasiprojective variety is essentially of finite type.

\begin{prop}[{\cite[Theorem 2]{AIgorenstein}}]
    Let $A$ be a commutative $k$-algebra, essentially of finite type. Then $A$ is Iwanaga--Gorenstein if and only if the $A$-module $$\bigoplus_n\ext_{A^e}^n(A,A^e)$$is invertible.
\end{prop}
\begin{cor}\label{gorcor}
Let $A$ be a commutative $k$-algebra, essentially of finite type.
\begin{enumerate}
\item If $A$ is CY then it is Iwanaga--Gorenstein.
\item If $A$ is irreducible and Iwanaga--Gorenstein then it is twisted CY.
\item If $A$ is irreducible and local with residue field $l$, then $A$ is CY if and only if $A\to l$ is Gorenstein.
\end{enumerate}
\end{cor}
\begin{proof}
For brevity let $E$ denote the module $\oplus_n\ext_{A^e}^n(A,A^e)$. For (1), if $A$ is CY then there is an isomorphism $E\cong A$, which is certainly an invertible module. For (2), since $A$ is irreducible, the regular module is indecomposable, and hence invertible modules are indecomposable. In particular if $E$ is invertible, there is a unique integer $t$ such that $\ext_{A^e}^t(A,A^e)\cong E$ and $\ext_{A^e}^i(A,A^e)$ vanishes for $i\neq t$. It follows that there is a quasi-isomorphism $A^\vee[t] \simeq E$ of $A$-bimodules, and hence $A^\vee$ is twisted CY. For (3), simply follow the proof of (2) and observe that an invertible module over a local ring is free of rank 1.
\end{proof}

In particular, a Gorenstein commutative algebra essentially of finite type is `locally nonsmooth CY'.

\begin{ex}\label{scy}
    Let $R$ be a commutative noetherian $k$-algebra of finite Krull dimension $d$. Recall from \cite{iyamareiten, iyamawemyss} that $R$ is said to be \textbf{singular CY} if it is Gorenstein and equicodimensional \cite[3.10]{iyamareiten}. An algebra essentially of finite type is noetherian and of finite Krull dimension, and an irreducible algebra is equicodimensional. Hence, if $R$ is irreducible and essentially of finite type, then it is singular CY if and only if it is Gorenstein, and \ref{gorcor} provides implications
    $$R \text{ nonsmooth CY }\implies R \text{ singular CY } \implies R \text{ twisted CY.}$$
\end{ex}

\subsection{Frobenius algebras}
Recall that a discrete finite dimensional graded algebra $A$ is \textbf{$d$-Frobenius} if there is an isomorphism $A^*\cong A[d]$ of $A$-modules (equivalently, $A^\circ$-modules). Such an algebra is equipped with a Nakayama automorphism $\nu$, and one has an $A$-bimodule isomorphism $A^*\cong A_\nu[d]$, where $A_\nu$ denotes $A$ with the usual bimodule structure on the left, and twisted by $\nu$ on the right. One then obtains a quasi-isomorphism $\HH_\bullet(A,A_\nu)^* \simeq \HH^\bullet(A)[d]$. A \textbf{symmetric} Frobenius algebra is one whose Nakayama automorphism is the identity, so if $A$ is symmetric one has the simpler description $\HH_\bullet(A)^* \simeq \HH^\bullet(A)[d]$. For a comprehensive reference on the classical theory of Frobenius algebras, see \cite[\S16]{lam}.

\p We can generalise these definitions straightforwardly to the dg world. Let $A$ be an algebra. Say that $A$ is \textbf{derived $d$-Frobenius} (or just \textbf{Frobenius} for short) if there is a quasi-isomorphism $A^* \simeq A[d]$ of $A^\circ$-modules. This is also known as being a \textbf{(derived) Poincar\'e duality algebra}, as in e.g.\ \cite{DGI}. Say that $A$ is \textbf{twisted symmetric} if $A^*$ is an invertible $A$-bimodule. Say that $A$ is \textbf{derived $d$-symmetric} (or just \textbf{symmetric} for short) if $A^*\simeq A[d]$ as $A$-bimodules.

\begin{rmk}
    This definition is a generalisation of Brav and Dyckerhoff's notion of \textbf{right Calabi--Yau} \cite{BD}. Indeed, if $A$ is a proper algebra, then homotopy classes of $A$-bilinear morphisms $A\to A^*$ are parameterised by $HH_\bullet(A)^*$. In particular, if $A$ is a symmetric algebra then we obtain a class $\omega \in HH_\bullet(A)^*$, and in addition $A$ is said to be \textbf{right Calabi--Yau} if $\omega$ lifts to a class in cyclic homology $HC_\bullet(A)^*$. In \cite{HRCY}, a symmetric algebra is called a \textbf{weak proper Calabi--Yau algebra}.
\end{rmk}

\begin{rmk}
    Unlike in the discrete case, a twisted symmetric algebra need not be Frobenius: if $A$ and $B$ are two finite dimensional twisted symmetric algebras, of different degrees, then the product $A\times B$ is twisted symmetric but not Frobenius.
\end{rmk}

\p If $A$ is a derived $d$-Frobenius algebra, then $A\simeq A^{**}$ as vector spaces. This holds if and only $A\in \rf(A)$. However, Frobenius algebras need not be proper, as the following examples show.
\begin{ex}\label{frobex}\hfill
    \begin{enumerate}
        \item The shifted dual numbers $k[x]/x^2$ with $x$ placed in degree $n$ is derived $n$-Frobenius. 
        \item If $V$ is the vector space $k\oplus k[1]$, then the square-zero extension $k\oplus V$ is not Frobenius, for dimension reasons.
        \item The graded field $K\coloneqq k[x,x^{-1}]$ with $x$ placed in nonzero degree $n$ is derived $m$-Frobenius whenever $m$ is an integer multiple of $n$. Note that $K$ does not admit an augmentation as a graded algebra.
        \item Let $V$ be the graded vector space with $V^i=k$ for all $i$. Then the square zero extension $k\oplus V$ is derived $0$-Frobenius.
    \end{enumerate}
\end{ex}

\begin{ex}[Symmetric completions]\label{symcomp}
    Recall that if $A$ is an algebra and $M$ a bimodule, one can form the \textbf{trivial extension algebra} $A\oplus M$ with multiplication given by $(a,m)(a',m')=(aa',am'+ma')$. In particular, let $A$ be a proper algebra. The $d$-\textbf{symmetric completion} of $A$ is the trivial extension algebra $T_d(A)\coloneqq A\oplus A^*[-d]$. One then has $$T_d(A)^* \cong A^*\oplus A[d] \simeq T_d(A)[d]$$showing that $T_d(A)$ is $d$-symmetric.
\end{ex}

\begin{rmk}
Let $A$ be a $d$-Frobenius algebra. Then its cohomology algebra $B\coloneqq H^*(A)$ is also $d$-Frobenius. If $A$ was augmented and local, then $B$ is local, and its augmentation ideal $\mathfrak{m}$ is nilpotent: since $B$ is local, $B^0$ is an Artinian local $k$-algebra, and in particular $\mathfrak{m}^0$ is nilpotent. If $x\in \mathfrak{m}$ is a non-nilpotent of nonzero degree $n$, then let $y\in\mathfrak{m}$ be the corresponding element of degree $d-n$ across the isomorphism $B\cong B^*[d]$. Then $yx=1 \in \mathfrak{m}$, which is a contradiction.
\end{rmk}

\begin{rmk}
Classically, one may also define Frobenius algebras in terms of a nondegenerate pairing. This has an analogue for derived Frobenius algebras: if $A$ is an algebra, say that a \textbf{Frobenius pairing} on $A$ is a bilinear map $\sigma:A\otimes_k A \to k$ such that \begin{enumerate}
    \item $\sigma$ is \textbf{homotopy nondegenerate}: $H^*\sigma$ is nondegenerate.
    \item $\sigma(ab,c) = \sigma(a,bc)$.
\end{enumerate}
Then an algebra $A$ is Frobenius if and only if it admits a Frobenius pairing: a left $A$-linear quasi-isomorphism $\phi:A \to A^*$ corresponds to a Frobenius pairing via the formula $\sigma(a,b)=\phi(b)(a)$.
\end{rmk}

\begin{prop}\label{nakex}
If $A$ is a derived $d$-Frobenius algebra, then it is twisted symmetric: there exists a homotopy automorphism $\nu$ of $A$ such that $A^*\simeq A_\nu[d]$ as $A$-bimodules.
\end{prop}
\begin{proof}
One can use the hom-tensor adjunction along with the fact that $A\simeq A^{**}$ to show that $\R\enn_A(A^*[-d])\simeq A$. So applying \ref{nakesistence} to the left $A$-module quasi-isomorphism $A\simeq A^*[-d]$, one obtains the desired homotopy automorphism.
\end{proof}
We call any such $\nu$ a \textbf{Nakayama automorphism} of $A$. Note that $A$ is a symmetric algebra if and only if $\id$ is a Nakayama automorphism. 

\p Augmented Frobenius algebras are Gorenstein:

\begin{prop}[cf.\ {\cite[7.2]{DGI}}]\label{frobisgor}
    Let $A$ be a $d$-Frobenius algebra and $g:A \to k$ an augmentation. Then $g$ is Gorenstein. 
\end{prop}
\begin{proof}
The natural quasi-isomorphism $A \to \R\enn_A(A^*[-d])$ is a morphism of augmented algebras, so by \ref{augnakrmk} we may choose the Nakayama automorphism $\nu$ of $A$ to be a homotopy automorphism of $A$ in the category of augmented algebras. In particular, it follows that $k_\nu\simeq k$ as $A$-bimodules. We hence have natural quasi-isomorphisms of $A$-bimodules
\begin{align*}
    k&\simeq\R\hom_A(k,A^*)&\text{by hom-tensor}\\
    &\simeq \R\hom_A(k,A_\nu[d])&\text{by \ref{nakex}}\\
    &\simeq \R\hom_A(k_{\nu^{-1}},A)[d]&\\
    &\simeq \R\hom_A(k,A)[d]&\text{since }\nu \text{ fixes }k
\end{align*}as required.
\end{proof}

\begin{rmk}
    One can define an algebra to be \textbf{fractional Frobenius} if there exists some $m>0$ such that $(A^*)^{\lot_A m} \simeq A[n]$. In view of \ref{pfistwcy} below, one can prove that a proper algebra is fractional Frobenius if and only if it is fractional twisted CY, as in \ref{fracCYrmk}. In particular, the references given there provide many examples of fractional Frobenius algebras.
\end{rmk}

\begin{rmk}\label{symhh}
    If $A$ is any algebra, the hom-tensor adjunction yields a quasi-isomorphism $\HH_\bullet(A)^* \simeq \HH^\bullet(A,A^*)$, and in particular if $A$ is $d$-symmetric then we have a quasi-isomorphism $\HH_\bullet(A)^* \simeq \HH^\bullet(A)[d]$.
\end{rmk}

\subsection{Proper Frobenius algebras}
In this section we study proper Frobenius and twisted symmetric algebras in detail. The latter admit multiple characterisations:
\begin{lem}\label{propfrobdict}
    Let $A$ be a proper algebra. The following are equivalent:
    \begin{enumerate}
        \item $A$ is twisted symmetric.
        \item The right module $A^*$ is a thick generator of $\per A$.
        \item The functor $-\lot_A A^*$ is an autoequivalence of $\per A$.
        \item The functor $-\lot_A A^*$ is a Serre functor on $\per A$.
        \item $\per A$ admits a Serre functor.
    \end{enumerate}
\end{lem}
\begin{proof}Noting that $A$ is twisted symmetric if and only if $A^*$ is invertible, the equivalence of (1), (2) and (3) follows from \ref{pregor} together with the easy observation that $\R\enn_A(A^*)\simeq A$. Observe that if $P \in \per A$ then we have $P\lot_A A^* \simeq \R\hom_A(P^\vee, A^*)\simeq P^{\vee *}$ by hom-tensor. So for all $P\in \per A$ and $M\in D(A)$ we have $\R\hom_A(M,P\lot_A A^*) \simeq \R\hom_A(P,M)^*$ by hom-tensor again. In particular, $-\lot_A A^*$ is a Serre functor on $\per A$ if and only if it is an autoequivalence of $\per A$, which shows that (3) and (4) are equivalent. Clearly (4) implies (5), and the converse is \cite[5.5]{goodbody}. 
\end{proof}

\begin{rmk}
    If $A$ is a finite dimensional discrete algebra then $A^*$ is an invertible bimodule in the classical sense (i.e.\ invertible with respect to the underived tensor product $\otimes_A$) if and only if $A$ is self-injective.
\end{rmk}

\begin{ex}
    Let $A$ be a proper algebra such that $\per A$ admits a Serre functor $\mathbb{S}$. Then $A$ is a Frobenius algebra if and only if it is a Calabi--Yau object; i.e.\ $\mathbb{S}A \simeq A[d]$.
\end{ex}

\begin{ex}
    Let $A$ be a smooth proper algebra. Then it is well known that $-\lot_A A^*$ is a Serre functor on $\per A$ \cite{shklyarov}. In particular, $A$ is twisted symmetric.
\end{ex}

\begin{ex}[Algebraic geometry]\label{agex}
    Let $X$ be a proper irreducible $k$-scheme. Grothendieck duality yields natural quasi-isomorphisms $$\R\hom_X(\mathcal{F},\mathcal{G}\lot_X \omega_X^\bullet) \simeq \R\hom_X(\mathcal{G},\mathcal{F})^*$$whenever $\mathcal{F} \in D(X)$ is coherent and $\mathcal{G} \in \per X$ is perfect, where $\omega_X^\bullet$ denotes the dualising complex. In particular $\per X$ admits a Serre functor. Note that $X$ is Cohen--Macaulay precisely when $\omega_X^\bullet$ is a shift of a sheaf, and is Gorenstein precisely when it is a shifted line bundle, in which case we have $\omega_X^\bullet = \omega_X[\dim X]$. By \cite{bvdb}, $D(X)$ has a compact generator, and in particular there exists some proper algebra $A$ with $\per A \simeq \per X$. By \ref{propfrobdict}, we see that $A$ is twisted symmetric. It is symmetric precisely when the Serre functor is a shift, i.e.\ $X$ has trivial canonical bundle (this is the geometric meaning of `nonsmooth Calabi--Yau'). In this setting, one may use \ref{symhh} to determine $\HH^*(X)=\HH^*(A)$ in terms of $\HH_*(X)=\HH_*(A)$, which may be computed via the HKR theorem when applicable.
\end{ex}
 The product of two proper twisted symmetric algebras is proper twisted symmetric:
\begin{lem}\label{frobtenslem}
    Let $A$ and $B$ be proper twisted symmetric algebras. Then $A\otimes B$ is a proper twisted symmetric algebra.
\end{lem}
\begin{proof}
    Putting $P\coloneqq A\otimes B$ for brevity, we have $P^*\simeq A^*\otimes B^*$ as $P$-bimodules, where the first quasi-isomorphism uses that $A$ is proper. Since the tensor product of an invertible $A$-bimodule with an invertible $B$-bimodule is an invertible $P$-bimodule, we are done.
\end{proof}
\begin{rmk}
    It is not the case that a tensor product of non-proper Frobenius algebras is Frobenius. For example, take $A=k[x,x^{-1}]$ with $|x|=1$. Then $A\otimes A$ is not reflexive and so cannot be Frobenius.
\end{rmk}

\begin{prop}\label{TwFrobIsTwCY}
    Let $A$ be a proper algebra. Then the following are equivalent:
    \begin{itemize}
        \item $A$ is twisted symmetric.
        \item $\per A$ admits a twisted Calabi--Yau structure.
    \end{itemize}
    Moreover, $A$ is $d$-symmetric if and only if $\per A$ admits a $d$-Calabi--Yau structure.
\end{prop}

\begin{proof}
If (1) holds then $-\lot_A A^*$ is a Serre functor on $\per A$ by \ref{propfrobdict}, and hence (2) holds. Conversely, if (2) holds then, the naturality of the CY structure isomorphisms implies that we have an invertible $A$-bimodule $X$ and an $A$-bimodule quasi-isomorphism $$A^*\simeq \R\enn_A(A,A)^*\simeq \R\enn_A(A,X)\simeq X$$and so (1) holds. Note that this direction did not use the properness hypothesis on $A$. The final statement follows by taking $X=A[d]$.
\end{proof}

Combining \ref{TwFrobIsTwCY} with \ref{nakex}, we obtain:
\begin{cor}\label{PerFrobIsTwCY}
    Let $A$ be a proper derived $d$-Frobenius algebra, with Nakyama automorphism $\nu$. Then $\per A$ admits a twisted Calabi--Yau structure, with twisting bimodule $A_\nu$.
\end{cor}

We next show that a proper algebra is twisted symmetric if and only if it is twisted CY. The following lemma is key:

\begin{lem}\label{proplem}
    Let $A$ be a proper algebra. Then there is an $A$-bilinear quasi-isomorphism $A^\vee \simeq \R\hom_A(A^*,A)$.
\end{lem}
\begin{proof}
    Since $\R\hom_k(A^*,A)\simeq A\otimes A$ is a perfect $A$-bimodule, applying \ref{kellerlem} with $M=A^*$ and $L=A$ (and using $M^{*}\simeq A$) gives the required $A$-bilinear quasi-isomorphism $A^\vee \to \R\hom_A(A^*,A)$.
\end{proof}

\begin{prop}\label{pfistwcy}
    Let $A$ be a proper algebra. Then $A$ is {twisted symmetric} if and only if it is twisted Calabi--Yau.
\end{prop}
\begin{proof}
If $A$ is twisted symmetric then $A^*$ is an invertible bimodule. Moreover, $\R\hom_A(A^*,A)$ is its inverse, and hence itself invertible. Then \ref{proplem} shows that $A^\vee$ is invertible, and hence $A$ is twisted CY. The converse implication is similar.
\end{proof}

\begin{cor}\label{cyimpsym}
     Let $A$ be a proper algebra. Then $A$ is symmetric if and only if it is Calabi--Yau.
\end{cor}

\begin{ex}
    Suppose that $X$ is a smooth projective variety over $k$. Let $A$ be a smooth proper algebra derived equivalent to $X$, in the sense that there is an equivalence $D^b(\mathbf{coh}X)\simeq \per A$ (the existence of such an $A$ is assured by \cite{bvdb}). Then $X$ is Calabi--Yau if and only if $A$ is: to see this, observe that \ref{agex} shows that $X$ is CY if and only if $A$ is symmetric, and the claim now follows from \ref{cyimpsym}.
\end{ex}

For a proper augmented algebra, the Gorenstein and Frobenius conditions are closely related. Indeed in \cite{jin}, a proper algebra satisfying any of the equivalent conditions of \ref{propfrobdict} is called \textbf{Gorenstein}; we refrain from using this terminology. We have already shown that a proper Frobenius algebra is Gorenstein, and a result of Goodbody gives a partial converse:

\begin{prop}\label{gorimpfrob}
Let $A$ be a finite dimensional local algebra. If both $A$ and $A^\circ$ are Gorenstein, then $A$ is Frobenius.
\end{prop}
\begin{proof}
Follows from \cite[5.8]{goodbody}.
\end{proof}

\section{Calabi--Yau, Gorenstein, and Frobenius coalgebras}
We define the above notions for coalgebras, analogously to the definitions for algebras, and prove our main theorem. We also study how Frobenius coalgebras behave across linear duality and finish with some discussion of endomorphism algebras.

\subsection{The main theorem}

Let $C$ be a coalgebra and $M$ a right $C$-comodule. Recall that the left $C$-comodule $M^*$ is the `linear dual' of $M$, which is $\hom_k(M,k)$ when $M$ is finite dimensional and in general is extended via homotopy (co)limits in $\Comod C$. Recall also that when $M$ is a $C$-bicomodule we write $D(M)$ for its `linear dual', which is $\hom_k(M,k)$ when $M$ is finite dimensional and in general is extended via homotopy (co)limits, this time in $C\bicomod C$. We also use the same notation for bicontramodules across the co-contra correspondence.

\p If $C$ is a coalgebra we denote by $C^\vee$ the $C$-bicontramodule $\R\hom^{C^e}(C^*,(C^e)^*)$. 

\begin{defn}\label{cogmaindef}
Let $C$ be a coalgebra.
    \begin{enumerate}
        \item Say that $C$ is \textbf{twisted Calabi--Yau} if there is an invertible bicomodule $X$ such that $C^* \simeq {}_X(C^\vee)$. Say that $C$ is \textbf{$d$-Calabi--Yau} if there is a weak equivalence $C^*\simeq C^\vee[d]$ of bicontramodules.
        \item Say that $C$ is $d$-\textbf{Gorenstein} if there is a weak equivalence $\R\hom^C(k,C^*)[d]\simeq k$ of left $C$-contramodules.
        \item Say that $C$ is $d$-\textbf{Frobenius} if there is a weak equivalence $C^*[d]\simeq C$ of left $C$-comodules.
        \item Say that $C$ is \textbf{twisted symmetric} if $D(C)$ is an invertible bicomodule. Say that $C$ is $d$-\textbf{symmetric} if there is a weak eqivalence $D(C)[d]\simeq C$.
    \end{enumerate}
\end{defn}

\begin{rmk}
    The definitions of Frobenius and (twisted) symmetric use comodules as these are best adapted to `$k$-linear duality' in the coalgebraic setting. On the other hand, the definitions of Gorenstein and (twisted) CY use contramodules as these are best adapted to `$C$-linear duality'. Using the co-contra correspondence, it is possible to state all of the above definitions entirely in terms of comodules or contramodules.
\end{rmk}

If $A$ is a pseudocompact algebra, then we use the same terminology for $A$, with the understanding that it applies to the coalgebra $A^*$.
\begin{thm}\label{mainthm} Let $(C,A)$ be a Koszul duality pair. 
\begin{enumerate}
        \item $A$ is twisted Calabi--Yau if and only if $C$ is twisted symmetric.
        \item $A$ is $d$-Calabi--Yau if and only if $C$ is $d$-symmetric.
        \item $A$ is $d$-Gorenstein if and only if $C$ is $d$-Frobenius.
        \item $A$ is $d$-Frobenius if and only if $C$ is $d$-Gorenstein.
        \item $A$ is twisted symmetric if and only if $C$ is twisted Calabi--Yau.
         \item $A$ is $d$-symmetric if and only if $C$ is $d$-Calabi--Yau.
    \end{enumerate}
\end{thm}
\begin{proof}For all statements, the proof consists of matching up appropriate modules, co/contramodules, and functors, across the relevant version of Koszul duality. We begin with (1). By \ref{diagonal}, across bimodule Koszul duality the diagonal bimodule $A$ corresponds to the diagonal bicomodule $C$. Hence by \ref{CYalgSFcog}, across bimodule Koszul duality the inverse dualising complex $A^\vee$ corresponds to $D(C)$. So one is invertible if and only if the other is, which proves (1). The proof of (2) is identical.
    
     For (3), \ref{posKDthm}(5) shows that across one-sided Koszul duality the left $A$-module $k$ corresponds to the left $C$-comodule $C$. Hence \ref{GalgFcog} shows that the left $A$-module $\R\hom_A(k,A)$ corresponds to the left $C$-comodule $C^*$, and now it follows that $A$ is $d$-Gorenstein if and only if $C$ is $d$-Frobenius. 
    
    For (4), \ref{posKDthm}(6) shows that across one-sided Koszul duality the left $C$-contramodule $k$ corresponds to the left $A$-module $A^*$. Hence \ref{onesd} shows that the left $C$-contramodule $\R\hom^C(k,C^*)$ corresponds to the left $A$-module $A^{**}$. Hence if $A$ is $d$-Frobenius, we have $A^*\simeq A^{**}[d]$ as left $A$-modules and hence $C$ is $d$-Gorenstein. In the other direction, if $C$ is $d$-Gorenstein then the left $A$-module $M\coloneqq A^*$ is isomorphic to $M^*[d]$, which implies that $M\in\rf{A}$ and hence $A\in \rf{A}$. Hence $A^{**}\simeq A$ and so A is $d$-Frobenius.

    For (5), \ref{codiagonal} shows that across two-sided Koszul duality the bimodule $A^*$ corresponds to the bicontramodule $C^*$. Hence \ref{twosd} implies that the $A$-bimodule $A^{**}$ corresponds to the $C$-bicontramodule $C^\vee$. If $A$ is twisted symmetric, we see that there is an invertible bimodule $X$ such that $A^* \simeq {}_X(A^{**})$. Translating this across Koszul duality, we see that there is an invertible bicomodule $Y$ such that $C^*\simeq {}_Y(C^\vee)$; in other words $C$ is twisted CY. Conversely if $C$ is twisted CY then with we obtain an invertible $A$-bimodule $X$ and an $A$-bilinear quasi-isomorphism $A^{*}\simeq {}_X(A^{**})$, and so putting $M\coloneqq A^*$ we have $A$-bilinear quasi-isomorphisms
    \begin{align*}
        M&\simeq X\lot_A(M^*)&\\
        &\simeq X\lot_A\left((X\lot_A(M^*))^*\right)& \text{substituting the first line into itself}\\
        & \simeq X\lot_A\R\hom_A(X,M^{**}) & \text{by hom-tensor}\\
        & \simeq M^{**} & \text{since $\R\hom_A(X,-)$ is inverse to $X\lot_A-$}
    \end{align*}
    and so $M$, and hence also $A$, is reflexive. Hence $A$ is twisted symmetric. The proof of (6) is identical but one takes $X=A[d]$ everywhere.
\end{proof}

\begin{ex}
Let $A$ be a commutative Gorenstein ring, essentially of finite type, and $A \to k$ any choice of augmentation on $A$. Since $A$ is $d$-Calabi--Yau for $d=\dim(A)$, it follows from \ref{mainthm} that the Koszul dual of $A$ is a $d$-symmetric coalgebra.
\end{ex}

\begin{ex}
    Let $A$ be a proper augmented algebra and $T_d(A)$ its $d$-symmetric completion, as in \ref{symcomp}. It follows from \ref{mainthm} that the coalgebra $BT_d(A)$ is $d$-Calabi--Yau, c.f.\ \cite{HLW}.
    \end{ex}

We can now improve on \ref{funnydual}, using the following lemma.

\begin{lem}\label{gorprod}\phantom{}\hfill
\begin{enumerate}
    \item Let $R,S$ be two Gorenstein augmented algebras. If $R$ is proper then $R\otimes S$ is a Gorenstein algebra.
    \item Let $C,D$ be two Frobenius coalgebras. If $C$ is regular then $C\otimes D$ is a Frobenius coalgebra.
\end{enumerate}
\end{lem}
\begin{proof}
    For (1), suppose that $R$ is proper $r$-Gorenstein and $S$ is $s$-Gorenstein. We compute
    \begin{align*}
        \R\hom_{R\otimes S}(k,R\otimes S)&\simeq \R\hom_R(k,\R\hom_S(k,R\otimes S))&\text{by hom-tensor}\\
        &\simeq \R\hom_R(k,R\otimes\R\hom_S(k,S)) &\text{since }R\text{ is proper}\\
        &\simeq \R\hom_R(k,R[-s])& \text{since }S\text{ is Gorenstein}\\
        &\simeq k[-r-s] & \text{since }R\text{ is Gorenstein}
    \end{align*}
    and hence $R\otimes S$ is $(r+s)$-Gorenstein. For (2), suppose that $C$ is regular $r$-Frobenius and $D$ is regular $s$-Frobenius. By \cite[5.1]{HLdg} we have a natural algebra quasi-isomorphism $\Omega(C\otimes D)\simeq \Omega (C) \otimes \Omega (D)$. Putting $R\coloneqq \Omega C$ and $S\coloneqq \Omega D$ we see by \ref{mainthm} that $R$ is $r$-Gorenstein and $S$ is $s$-Gorenstein. Moreover $R$ is proper since $C$ is regular. Hence by (1), $R\otimes S$ is $(r+s)$-Gorenstein, and hence $C\otimes D$ is $(r+s)$-Frobenius by \ref{mainthm} again.
\end{proof}
\begin{cor}\label{pwfrobcor}
    The pseudocompact algebra $A\coloneqq k\llbracket x_1,\ldots, x_n\rrbracket$ with $x_i$ in degree $d_i$ is $(n-\sum_id_i)$-Frobenius.
\end{cor}
\begin{proof}
    Write $A$ as an iterated completed tensor product $A\cong k\llbracket x_1\rrbracket\hat\otimes\cdots\hat\otimes k\llbracket x_n \rrbracket$ and observe that the $i^\text{th}$ tensorand is $(1-d_i)$-Frobenius by \ref{funnydual}. Since each tensorand is also regular, the result now follows from \ref{gorprod}(2).
\end{proof}

\begin{rmk}
    The number $(n-\sum_id_i)$ appearing in the previous theorem is to be thought of as the virtual dimension of the formal stacky derived scheme $A$; in particular $A$ should be `derived lci' in an appropriate sense.
\end{rmk}

\begin{rmk}
    Since bimodule Koszul duality is monoidal, a fractional version of \ref{mainthm} also holds; we leave the formulation to the reader.
\end{rmk}

\begin{rmk}\label{gorrmk}
    In fact, the discrete case of \ref{pwfrobcor} can be significantly generalised: we claim that if $R$ is a discrete commutative complete local augmented Gorenstein $k$-algebra of Krull dimension $d$, then - equipped with its natural topology - the pseudocompact $k$-algebra $R$ is $d$-Frobenius, i.e.\ we have a weak equivalence $R^* \simeq R[d]$ of pseudocompact $R$-modules. The following argument to show this is due to Benjamin Briggs and is a generalisation of the computation of \ref{funnydual}. Firstly, let $a_1,\ldots,a_d$ be a system of parameters for $R$, and put $I_i\coloneqq (a_1^i,\ldots, a_d^i)$, so that we have an isomorphism $R\cong \varprojlim_i R/I_i$. Each of the quotients $R/I_i$ is a zero-dimensional Gorenstein $k$-algebra (since a Gorenstein local ring is approximately Gorenstein \cite{hochster}), and hence $0$-Frobenius, so we have an isomorphism $(R/I_i)^* \cong R/I_i$ as $R$-modules. The linear dual $R^*$ is by definition $\hocolim_i (R/I_i)^* \simeq \hocolim_i(R/I_i)$, where we take the homotopy colimit in the category of pseudocompact $R$-modules. We may resolve $R/I_i$ by the Koszul complex $K_i\coloneqq K^R(a_1^i,\ldots, a_d^i)$, so that $R^*$ becomes $\varinjlim_i K_i$, where again the colimit is taken in pseudocompact $R$-modules. As in \ref{funnydual} we may compute this colimit as the pseudocompactification of the colimit taken in discrete $R$-modules instead. As in \cite[2.B]{firststeps}, this latter colimit is precisely the stable Koszul complex $K\coloneqq K_\infty^R(a_1,\ldots, a_d)$, which is the totalisation of a $d$-dimensional cube whose vertices are given by the localisations $R[a_{i_1}^{-1},\ldots a_{i_n}^{-1}]$, where the $i_j$ range over all possible subsets $J$ of $\{1,\ldots, d\}$. For example, if $R=k\llbracket x\rrbracket$ then the stable Koszul complex is the localisation map $k\llbracket x\rrbracket \to k\llbracket x\rrbracket[x^{-1}]$ with the rightmost term placed in degree zero, which appears in \ref{funnydual} (note that $\{1\}$ has precisely two subsets!). If $J$ is nonempty, then the corresponding localisation $R[a_{i_1}^{-1},\ldots a_{i_n}^{-1}]$ has no finite dimensional $R$-module quotients and hence its pseudocompactification is trivial. It follows that the pseudocompactification of $K$ is $R[d]$, as desired. We remark that the assumption that the residue field of $R$ is $k$ may be dropped at the cost of possibly replacing $k$ by a larger field, by the Cohen structure theorem. It is possible that this argument generalises to connective commutative complete local dg algebras.
\end{rmk}

\subsection{Proper Frobenius coalgebras}
We study how the Frobenius properties on algebras and coalgebras interact with linear duality. In general, a proper Frobenius coalgebra will dualise to a Frobenius algebra, but the converse need not be true; we identify a large class of coalgebras where this is however the case.

\begin{defn}\label{sccd}
    Say that a coalgebra $C$ is \textbf{simply connected} if it is connective, $C^0\cong k$, and $C_1 \cong 0$. 
\end{defn}
This notion is also known as \textbf{1-connected}; we remark that two simply connected coalgebras are weakly equivalent if and only if they are quasi-isomorphic. The following lemma may be of independent interest.

\begin{lem}\label{sccoglem}
    Let $C$ be a proper simply connected coalgebra. Then the natural localisation $\dco(C) \to D(\mathbf{Comod-}C)$ is an equivalence.
\end{lem}
\begin{proof}
    By \ref{qisoreflec}, it is enough to show that every object of $\dco(C)$ is $C$-colocal. If $A\coloneqq \Omega C$ denotes the Koszul dual of $C$, this is equivalent to showing that every object of $D(A)$ is $k$-colocal. Observe that since $C$ is simply connected, $A$ is a connective algebra with $A^0\cong k$. Then \ref{connloc} now shows that $A$ itself is $k$-colocal, and the claim follows.
\end{proof}
\begin{rmk}
    If $C$ is a proper simply connected coalgebra, then $\dco(C)$ admits a t-structure with heart the category of $k$-vector spaces: to see this, observe that since $\Omega C$ is connective with $H^0(\Omega C) \cong k$, the standard t-structure on $D(\Omega C)$ has the desired heart.
\end{rmk}

\begin{prop}\label{propLD}
    Let $C$ be a proper coalgebra and $A=C^*$ its linear dual algebra. 
    \begin{enumerate}
        \item If $C$ is $d$-Frobenius then $A$ is $(-d)$-Frobenius.
        \item If $A$ is $(-d)$-Frobenius and $C$ is simply connected, then $C$ is $d$-Frobenius.
    \end{enumerate}
    If in addition $C$ is strongly proper, the following hold:
    \begin{enumerate}
        \item[3.] If $C$ is $d$-symmetric then $A$ is $(-d)$-symmetric.
        \item[4.] If $A$ is $(-d)$-symmetric and $C$ is simply connected, then $C$ is $d$-symmetric.
    \end{enumerate}
\end{prop}
\begin{proof}
    Since $C$ is proper, it is weakly equivalent to a finite dimensional comodule over itself, and hence the comodule $C^*$ is the usual linear dual of $C$. By dualising, it follows that $A$ is $(-d)$-Frobenius if and only if $C$ is quasi-isomorphic to $C^*[d]$ as $C$-comodules. Since weak equivalences are quasi-isomorphisms, $(1)$ is now clear. The second claim follows similarly from \ref{sccoglem}. The third claims are proved identically; for (4) only needs to observe in addition that if $C$ is simply connected then so is $C^e$.
\end{proof}

\begin{rmk}
    The simply connectedness assumptions of the above proposition cannot be dropped. Counterexamples can be found in chain coalgebras of topological spaces; see \ref{scex} for the details.
\end{rmk}

\begin{rmk}
One can use the above to show that a strongly proper simply connected Frobenius coalgebra admits a Nakayama automorphism. On the Koszul dual side, this shows that a smooth connected Gorenstein algebra with degree zero part $k$ is twisted (Ginzburg) Calabi--Yau.
\end{rmk}

\begin{rmk}\label{BDmorita}
    Let $A$ be a local Ginzburg CY algebra. As in \ref{moritaduality}, the Koszul dual $A^!$ is a model for the Morita dual $\pvd A$ of $A$. Since $A$ is smooth, $BA$ is a proper Frobenius coalgebra, and hence $A^!$ is a Frobenius algebra. This generalises the one-object case of Brav and Dyckerhoff's theorem stating that the Morita dual of a left CY dg category is a right CY dg category \cite{BD}.
\end{rmk}

\section{Poincar\'e duality}
We discuss Gorenstein duality and Van den Bergh duality for regular Gorenstein and Ginzburg CY algebras respectively; in particular we view Van den Bergh duality as a `two-sided' version of Gorenstein duality. As an application we show that a topological space is a Poincar\'e duality space if and only if its coalgebra of chains is a Frobenius coalgebra. In the simply connected case this recovers F\'elix--Halperin--Thomas's notion of Gorenstein space \cite{FHT}. We use some facts about dg Hopf algebras to recover Brav and Dyckerhoff's version of Poincar\'e duality \cite{BD}. We give some applications from string topology, rational homotopy theory, and Lie theory.

\subsection{Gorenstein duality}
We formulate a notion of Gorenstein duality for regular augmented algebras $A$, and show that such an $A$ has Gorenstein duality if and only if it is Gorenstein.

Let $A$ be an augmented algebra. We put $\zeta_A\coloneqq \R\hom_A(k,A)$ and regard $\zeta_A$ as a left $A$-module. We think of $\zeta_A$ as a `one-sided dualising complex' for $A$. 

\p Let $M$ be a right $A$-module. The \textbf{Gorenstein homology} of $A$ with coefficients in $M$ is the vector space $\hg_M\coloneqq M\lot_A k$ and the \textbf{Gorenstein cohomology} of $A$ with coefficients in $M$ is the vector space $\hg^M\coloneqq \R\hom_{A}(k,M)$. 

\begin{lem}\label{pregordu}Let $A$ be a regular augmented algebra and $M$ an $A$-module. There is a natural vector space quasi-isomorphism $\hg^M\simeq M\lot_A \zeta_A$.
\end{lem}
\begin{proof}
    Regularity of $A$ means that $k$ is perfect, and so $\R\hom_A(k,-)$ is naturally isomorphic to tensoring with the right dual $\zeta_A$ of $k$.
\end{proof}

If $A$ is a regular augmented algebra, observe that there is a natural quasi-isomorphism $\hg_k \simeq \R\hom_{A^\circ}(\zeta_A,k)$. This allows us to view classes in $H^m(\hg_k)$ as (homotopy classes of) left $A$-linear maps $\zeta_A \to k[m]$. 

\begin{defn}\label{gordudef}
    Let $A$ be a regular augmented algebra and $M$ an $A$-module. For a fixed $\theta \in H^m(\hg_k)$, the \textbf{cap product with} $\theta$ is the induced map of vector spaces $$-\cap\theta:\hg^M \simeq M\lot_A \zeta_A \xrightarrow{M\lot_A\theta} M \lot_A k[m] \simeq \hg_M[m].$$Say that $A$ \textbf{has $n$-Gorenstein duality} if there is a class $\theta \in H^{-n}(\hg_k)$ such that, for each $M$, the map $-\cap\theta$ is a quasi-isomorphism.
\end{defn}
The minus sign in the definition of $n$-Gorenstein duality is motivated by the following theorem:

\begin{prop}\label{gorduprop}
    Let $A$ be a regular augmented algebra. Then $A$ is $n$-Gorenstein if and only if $A$ has $n$-Gorenstein duality.
\end{prop}
\begin{proof}
    Observe that $A$ has $n$-Gorenstein duality if and only if there is a class $\theta$ such that the corresponding left $A$-linear map $\zeta_A \to k[-n]$ is a quasi-isomorphism; this latter condition is equivalent to $A$ being $n$-Gorenstein.
\end{proof}

\begin{rmk}
If $A$ is regular Gorenstein, then since $\hg_k\simeq BA$ is (quasi-isomorphic to) a Frobenius coalgebra, the class $\theta$ witnessing Gorenstein duality is necessarily the image of the counit under the weak equivalence $BA \simeq (BA)^*[d]$.
\end{rmk}

\subsection{Van den Bergh duality}
We recall some of the main results of \cite{vdbduality} in a manner similar to our above Gorenstein duality framework. Let $A$ be an algebra. We will write $\xi_A$ for the inverse dualising complex $\R\hom_{A^e}(A,A^e)$. If $M$ is an $A$-bimodule we will put $\hh_M\coloneqq \HH_\bullet(A,M)$ and $\hh^M\coloneqq \HH^\bullet(A,M)$.

\begin{lem}\label{prevdbdu}Let $A$ be a smooth algebra and $M$ an $A$-bimodule. There is a natural vector space quasi-isomorphism $\hh^M\simeq M\lot_{A^e} \xi_A$.
\end{lem}
\begin{proof}
    Regularity of $A$ means that $A$ is a perfect $A$-bimodule, and so $\R\hom_{A^e}(A,-)$ is naturally isomorphic to tensoring with the dual $\xi_A$ of $A$.
\end{proof}
If $A$ is smooth, then there is a natural quasi-isomorphism $\hh_A \simeq \R\hom_{A^e}(\xi_A,A)$ and hence we can view classes in $H^m(\hh_A)$ as homotopy classes of $A$-bilinear maps $\xi_A \to A[m]$. 

\begin{defn}\label{vdbdudef}
    Let $A$ be a smooth algebra and $M$ an $A$-bimodule. For a fixed $\eta \in H^m(\hh_A)$, the \textbf{cap product with} $\eta$ is the induced map of vector spaces $$-\cap\eta:\hh^M \simeq M\lot_{A^e} \xi_A \xrightarrow{M\lot_{A^e}\eta} M \lot_{A^e} A[m] \simeq \hh_M[m].$$Say that $A$ \textbf{has $n$-Van den Bergh duality} if there is a class $\eta \in H^{-n}(\hh_A)$ such that, for each $M$, the map $-\cap\eta$ is a quasi-isomorphism.
\end{defn}

\begin{prop}\label{vdbduprop}
    Let $A$ be a smooth algebra. Then $A$ is $n$-Calabi--Yau if and only if $A$ has $n$-Van den Bergh duality.
\end{prop}
\begin{proof}
Analogous to the proof of \ref{gorduprop}.
\end{proof}

\subsection{Hopf algebras}
Our reference for Hopf algebras is \cite{witherspoon}; although only discrete algebras are considered there, the relevant proofs will also work for dg algebras. Let $A$ be a dg Hopf algebra with comultiplication $\Delta$ and antipode $S$. We regard $A$ as an augmented algebra via the counit. Our goal is to prove the following theorem:

\begin{thm}\label{hopfthm}
    Let $A$ be a dg Hopf algebra with bijective antipode. Suppose that one of the following two conditions is satisfied:
    \begin{enumerate}
    \item $A^\circ$ is regular.
    \item $A$ is proper.
    \end{enumerate}
Then $A$ is a $d$-Calabi--Yau algebra if and only if it is a $d$-Gorenstein algebra.
\end{thm}

\begin{rmk}\label{hopfrmk}
If $A$ is a dg Hopf algebra, and $A'$ is a dg algebra with an algebra quasi-isomorphism $A'\simeq A$, then $A'$ is Calabi--Yau or Gorenstein precisely when $A$ is. In particular we may apply the above theorem to $A'$ verbatim. We will make use of this observation later.    
\end{rmk}

\begin{rmk}
    Heuristically, we think of dg Hopf algebras as Koszul dual to commutative coalgebras. In particular, since a proper commutative algebra is Frobenius if and only if it is symmetric, proving this heuristic would prove the theorem. However our proof will instead be a comparison of the Gorenstein and the Hochschild cohomology of $A$.
\end{rmk}

\begin{rmk} Finite dimensional Hopf algebras have bijective antipode. Recall that $A$ is said to be \textbf{involutive} if $S^2=\id$. Certainly an involutive Hopf algebra has bijective antipode. If $A$ is commutative or cocommutative then it is involutive; in particular group algebras and universal enveloping algebras have bijective antipode.
\end{rmk}

To begin the proof of \ref{hopfthm}, we need to define some auxiliary modules of interest. By the proof of \cite[9.4.1]{witherspoon}, both $\Delta:A \to A\otimes A$ and $\id \otimes S: A\otimes A \to A^e$ are algebra morphisms; write $\delta$ for their composition. We let $\mathcal{R}$ be the $A^e$-$A$-bimodule given by taking the diagonal $A^e$-bimodule and restricting the right action along $\delta$. Concretely, using Sweedler notation, the action is given by the formula $(a\otimes b).(x\otimes y).c = axc_1 \otimes S(c_2)yb$. Clearly $\mathcal{R}$ is free as a left $A^e$-module. Similarly, we write $\delta'$ for the composition $(S\otimes \id)\Delta$, and $\mathcal{L}$ for the analogous $A$-$A^e$-bimodule. We will sometimes wish to view $\mathcal{R}$ and $\mathcal{L}$ as $A$-bimodules, in which case we forget part of the $A^e$-action.

The following is our key intermediate lemma (cf.\ \cite[4.12]{HRCY} for a similar statement):

\begin{lem}\label{adlem}
Suppose that $A$ is a dg Hopf algebra with bijective antipode. Then:
\begin{enumerate}
    \item Up to quasi-isomorphism, $\mathcal{R}$ is free as a right $A$-module and $\mathcal{L}$ is free as a left $A$-module.
    \item There are $A$-bilinear quasi-isomorphisms $\mathcal{R}\lot_A k \simeq A \simeq k \lot_A \mathcal{L}$.
    \item There is an $A^e$-$A$-bilinear quasi-isomorphism $\R\hom_{A^e}(\mathcal{L},A^e)\simeq \mathcal{R}$ and an $A$-$A^e$-bilinear quasi-isomorphism $\R\hom_{{A^e}}(\mathcal{R},A^e)\simeq \mathcal{L}$.
\end{enumerate}

\end{lem}
\begin{proof}
We will only prove the desired statements for $\mathcal{R}$; the analogous statements for $\mathcal{L}$ then follow by symmetry. Claim (1) follows from the proof of \cite[9.2.9]{witherspoon}, which requires that $S$ be bijective (concretely, $\mathcal{R}$ is the free right $A$-module on the vector space $A$). Claim (2) follows from Claim (1) together with the fact that $\mathcal{R}\otimes_A k\cong A$, which is proved in \cite[9.4.2]{witherspoon} (and does not require any condition on the antipode). For claim (3), since $\mathcal{R}$ is free of rank one as a left $A^e$-module, we see that the $A$-$A^e$-bimodule $\R\hom_{{A^e}}(\mathcal{R},A^e)$ is simply given by $A^e$ where the right action is the regular $A^e$-action and the left action is given through the antipode $S$. But this is precisely $\mathcal{L}$, as desired.
\end{proof}

\begin{rmk}
    Let $A$ be a dg Hopf algebra and $B$ any dg algebra. If $M$ is a $B$-$A^e$-bimodule we put $M^\mathrm{ad}\coloneqq \R\hom_{A^e}(\mathcal{L},M)$, which is a $B$-$A$-bimodule. Observe that $M^\mathrm{ad}\simeq \hom_{A^e}(\mathcal{L},M)$ is simply $M$ but with a twisted $A$-action. The hom-tensor adjunction shows that when $A$ has bijective antipode, there is a natural $B$-$A$-bilinear quasi-isomorphism
$$\R\hom_{A}(k,M^\mathrm{ad})\simeq \R\hom_{A^e}(A, M)$$
comparing the Gorenstein cohomology of $M^\mathrm{ad}$ and the Hochschild cohomology of $M$. This is a dg version of a theorem of Ginzburg and Kumar \cite{gkumar}.
\end{rmk}

\begin{proof}[Proof of Theorem \ref{hopfthm}.]
Suppose first that we are in case (1): i.e.\ $A$ is a dg Hopf algebra with bijective antipode and $A^\circ$ is regular. By \ref{regCYisGor}, if $A$ is CY then it is Gorenstein. For the converse, suppose that $A$ is Gorenstein. Then we have natural $A$-bilinear quasi-isomorphisms
\begin{align*}
    A^\vee &\coloneqq \R\hom_{A^e}(A,A^e)&\\
    &\simeq \R\hom_{A^e}(\mathcal{R}\lot_A k,A^e)&\text{by \ref{adlem}(2)}\\
    &\simeq \R\hom_{A^\circ}(k,\R\hom_{A^e}(\mathcal{R},A^e))&\text{by hom-tensor}\\
    &\simeq \R\hom_{A^\circ}(k,\mathcal{L})&\text{by \ref{adlem}(3)}\\
    &\simeq \R\hom_A(k,A)\lot_A \mathcal{L}&\text{since }A^\circ\text{ is regular}\\
    &\simeq k[-d]\lot_A \mathcal{L}&\text{since }A\text{ is Gorenstein}\\
    &\simeq A[-d]&\text{by \ref{adlem}(2) again}\\
\end{align*}
and hence $A$ is $d$-Calabi--Yau, as desired. For case (2), we proceed similarly: suppose that $A$ is a proper Hopf algebra with bijective antipode. Firstly, note that $A$ is CY if and only if is symmetric, by \ref{cyimpsym}. In particular if $A$ is CY it is Frobenius, and hence Gorenstein by \ref{frobisgor}. For the converse, suppose that $A$ is Gorenstein. Then we have natural $A$-bilinear quasi-isomorphisms
\begin{align*}
    A^*&\simeq (\mathcal{R}\lot_A k)^* & \text{by \ref{adlem}(2)}\\
    &\simeq \R\hom_A(\mathcal{R},k) & \text{by hom-tensor}\\
     &\simeq \R\hom_A(\mathcal{R},\R\hom_A(k,A))[d] & \text{since }A\text{ is Gorenstein}\\
      &\simeq \R\hom_A(\mathcal{R}\lot_Ak,A)[d] & \text{by hom-tensor again}\\
      &\simeq A[d]& \text{by \ref{adlem}(2) again}
\end{align*}
and so $A$ is $d$-symmetric, and hence $d$-CY, as required.
\end{proof}

\begin{rmk}
In fact, one can show that a Gorenstein Hopf algebra $A$ with bijective antipode is Calabi--Yau: first write $k$ as a homotopy colimit of perfect modules to obtain a description of $\R\hom_A(k,A)$ as a homotopy inverse limit. Since this inverse limit is proper by hypothesis, it follows that it is actually equivalent to a finite inverse limit, and in particular it follows that the natural map $\R\hom_A(k,A)\lot_A \mathcal{L} \to \R\hom_{A^\circ}(k,\mathcal{L})$ is a quasi-isomorphism. Now one may run the proof of \ref{hopfthm} in case (1) to conclude that $A$ is CY. The same argument shows that if $A$ is a Calabi--Yau Hopf algebra with bijective antipode such that $\R\hom_A(k,A)$ is proper, then $A$ is Gorenstein. The authors are unaware of any dg Hopf algebra (with bijective antipode or not) which is CY and not Gorenstein.
\end{rmk}

We finish by giving an example. Recall that a Lie algebra $\mathfrak{g}$ has a universal enveloping algebra $U(\mathfrak{g})$, which is an involutive Hopf algebra. We let $CE(\mathfrak{g})$ denote the Chevalley--Eilenberg coalgebra of $\mathfrak{g}$, which computes its Lie algebra homology. In characteristic zero, these two objects are Koszul dual. Finally, recall that a Lie algebra $\mathfrak{g}$ is \textbf{unimodular} if every element $x\in \mathfrak{g}$ satisfies $\mathrm{tr}(\mathrm{ad}(x))=0$.

\begin{prop}\label{liethm}
Let $k$ be a field of characteristic zero and $\mathfrak{g}$ be a finite dimensional Lie algebra over $k$. Then the following are equivalent:
\begin{enumerate}
        \item $\mathfrak{g}$ is unimodular.
        \item $U\mathfrak{g}$ is a Calabi--Yau algebra.
        \item $U\mathfrak{g}$ is a Gorenstein algebra.
        \item $CE(\mathfrak{g})$ is a Frobenius coalgebra.
        \item $CE(\mathfrak{g})$ is a symmetric coalgebra.
    \end{enumerate}
\end{prop}

We remark that a quasi-isomorphism $CE(\mathfrak{g})\simeq CE(\mathfrak{g})^*[d]$ implies a Poincar\'e duality statement for the (co)homology of $\mathfrak{g}$; indeed one can view this proposition as a `commutative' version of our characterisation of Poincar\'e duality spaces.

\begin{proof}

Since $\mathfrak{g}$ is finite dimensional, $U(\mathfrak{g})$ has finite global dimension (equal to the dimension of $\mathfrak{g}$). The equivalence $(1) \iff (2)$ now follows from \cite[2.3]{HOZ} (ultimately, this result goes back to Hazewinkel \cite{hazewinkel}). Since $U(\mathfrak{g})$ has finite global dimension, it is regular, and hence $(2) \iff (3)$ follows from \ref{hopfthm}. Since we have  $CE(\mathfrak{g})\simeq B(U(\mathfrak{g}))$, the equivalences $(2) \iff (5)$ and $(3) \iff (4)$ follow from \ref{mainthm}.
\end{proof}

\begin{rmk}
    A more structured version of this example appears as \cite[Theorem 1.3]{HRCY}; since $U(\mathfrak{g})$ is regular one can lift the Poincar\'e duality class to a class in cyclic homology.
\end{rmk}

\begin{rmk}
    We believe that a similar result should hold if $\mathfrak{g}$ is assumed to be a finite dimensional $L_\infty$-algebra (for example, the minimal model of a proper dg Lie algebra). More accurately, we expect $(1)$ and $(3)$ to be equivalent in this setting; since $U(\mathfrak{g})$ need not be regular or proper we may not be able to apply \ref{hopfthm}. One can prove this in restricted settings: for example if $\mathfrak{g}$ is a finite dimensional graded Lie algebra such that $[x,y]=0$ for all $x$ of even degree and $y$ of odd degree, then $U(\mathfrak{g})$ decomposes as the tensor product of a regular algebra $U(\mathfrak{g}^\mathrm{even})$ and a finite dimensional algebra $U(\mathfrak{g}^\mathrm{odd})$. Analysing these pieces separately one can conclude that $U(\mathfrak{g})$ is Gorenstein if and only if $\mathfrak{g}$ is unimodular. To prove the desired claim in generality, one could try to generalise the main result of \cite{hazewinkel} to the $L_\infty$ setting; the techniques of \cite{BLunimod} should be relevant, especially the interpretation of unimodularity in terms of the Berezinian volume form.
\end{rmk}

\begin{rmk}
    A dg Hopf algebra can be thought of as a noncommutative derived group scheme; it would be interesting to interpret the results above geometrically.
\end{rmk}

\subsection{Poincar\'e duality spaces}
The material in this section is inspired by \cite[\S5.1]{BD}, which was itself inspired by \cite{luriesurgery}. Similar results appear in \cite[\S4.2]{HRCY}. Recall that a simplicial set $K$ is \textbf{grouplike} if its homotopy category $hK$ is a groupoid.
\begin{ex}
A Kan complex is grouplike, so in particular the singular simplicial set of a topological space is grouplike. In the sequel we will identify topological spaces and their singular simplicial sets.
\end{ex}
Let $C_\bullet(K)$ be the coalgebra of simplicial chains on $K$; when $K$ is a topological space, $C_\bullet(K)$ is the coalgebra of singular chains on $K$. We will also write $C^\bullet(K)$ for the linear dual of $C_\bullet(K)$, regarded as a plain dg algebra with no pseudocompact structure. Let $G(K)$ denote the Kan loop group of $K$; when $K$ is (the singular simplicial set of) a path connected topological space, $GK$ is homotopy equivalent to the based loop space $\Omega K$ and we will interchange these notions freely. Since $G(K)$ is a simplicial group, $C_\bullet(G(K))$ is an algebra under composition of loops. In fact, $C_\bullet(G(K))$ is an involutive dg Hopf algebra; the coalgebra structure is that of simplicial chains and the antipode is given by reversing loops. Note that when $K$ is connected, $C_\bullet(G(K))$ is connective.

\begin{rmk}\label{moorermk}
    If $X$ is a topological space and one uses Moore loops, then $C_\bullet(\Omega X)$ becomes a dg algebra, and there is a quasi-isomorphism $C_\bullet(\Omega X)\simeq C_\bullet (G\mathrm{Sing}X)$ of dg algebras. In the sequel we will tacitly assume that $C_\bullet(\Omega X)$ is a dg algebra in this manner. 
\end{rmk}

\begin{thm}[{\cite{chlmonoids,rzcubes}}]\label{loopgroup}
    If $K$ is a grouplike simplicial set, then there is a quasi-isomorphism $\Omega C_\bullet(K) \simeq C_\bullet G(K)$.
\end{thm}

\begin{rmk}
    For the topologically inclined reader, we emphasise that the dg coalgebra $C_\bullet Y$ depends on a choice of base field $k$, which we suppress from the notation. In particular, if $Y\to Y'$ is a $k$-homology isomorphism between simply connected spaces, then 
$C_\bullet Y \to C_\bullet Y'$ is a weak equivalence.
\end{rmk}

\begin{rmk}
    In general if $K$ is a connected simplicial set then $C_\bullet G(K)$ is a localisation of $\Omega C_\bullet(K)$ \cite[4.4]{chlmonoids}.
\end{rmk}

\begin{ex}
    Let $n>1$ and consider the coalgebra $C_\bullet S^n$ of chains on the $n$-sphere. This coalgebra is quasi-isomorphic to the cosquare-zero extension $k\oplus k[n]$, and since $S^n$ is simply connected it must in fact be weakly equivalent. Hence $C_\bullet S^n$ is a Frobenius coalgebra and its Koszul dual $C_\bullet \Omega S^n$ is a Gorenstein algebra. We will later see that both of these facts follow from the fact that $S^n$ is a Poincar\'e duality space. In fact, $C_\bullet S^1$ is also a Frobenius coalgebra for the same reason.
\end{ex}

If $X$ is a path connected topological space, then the derived category $D(C_\bullet \Omega(X))$ is equivalent to the $\infty$-category of $\infty$-local systems of vector spaces on $X$, since as in \cite{BD} it is identified with the category of $\infty$-functors from $X$ to $D(k)$. Since $X$ is canonically isomorphic to its opposite, the derived category of left $C_\bullet \Omega(X)$-modules is also identified with the category of $\infty$-local systems. One can recover the homology and cohomology of $X$ with values in a local system $M$ via the formulas $$H_\bullet(X,M)\simeq M\lot_Ak_X$$ $$H^\bullet(X,M)\simeq \R\hom_{A^\circ}(k_X,M)$$where $A=C_\bullet \Omega(X)$ and $k_X$ denotes the constant local system on $X$ with value $k$, which corresponds to a choice of augmentation $A\to k$, i.e.\ a basepoint. Since we work with path connected spaces the choice of a basepoint will not be important. In particular, $BA\simeq k\lot_A k$ is identified with $H_\bullet(X,k)$.

\p Say that a topological space is \textbf{finitely dominated} if it is a retract (in the homotopy category) of a finite CW complex. A finitely dominated space is homotopy equivalent to a CW complex \cite[A.11]{hatcher}.

\begin{ex}
A compact topological manifold is finitely dominated; in fact a compact topological manifold is homotopy equivalent to a finite CW complex by Kirby--Siebenmann.
\end{ex}

\begin{rmk}\label{wallfiniteness}
    A finitely dominated space is homotopy equivalent to a finite CW complex precisely when its Wall finiteness obstruction vanishes. In particular, a simply connected finitely dominated space is a finite CW complex.
\end{rmk}

\begin{rmk}
    A necessary condition for a CW complex $X$ to be finitely dominated is that its total integral homology $H_*(X,\Z)$ is a finitely generated abelian group. If $X$ is simply connected then this condition is sufficient \cite[4C.1]{hatcher}.
\end{rmk}

\begin{rmk}
    If $G$ is a group, $BG$ is typically \textit{not} finitely dominated.
\end{rmk}

The relevance of finitely dominated spaces for us will be the following proposition:
\begin{prop}\label{fdsmprop}
    Let $X$ be a path connected finitely dominated topological space. Then $C_\bullet \Omega(X)$ is smooth.
\end{prop}
\begin{proof}
    This is \cite[5.1]{BD}. The idea is as follows: suppose that $X$ is a homotopy retract of a finite CW complex $Y$, so that $C_\bullet \Omega(X)$ is a homotopy retract of $C_\bullet \Omega(Y)$. This latter algebra acquires a finite cell decomposition from that of $Y$, and this implies that $C_\bullet \Omega(X)$ is a finite type algebra in the sense of \cite{toenvaquie}, and hence smooth.\end{proof}

Let $X$ be a path connected finitely dominated topological space. Then as in \ref{gordudef}, given a class $\theta \in H_d(X,k)$ and an $A$-module $M$ we obtain a cap product map $$-\cap \theta: H^\bullet(X,M) \to H_{d-\bullet}(X,M)$$which agrees with the classical cap product map when $M=k$. Note that the sign conventions occur since we switch between homological and cohomological indexing.

\begin{defn}
    Let $X$ be a path connected finitely dominated topological space. Say that $X$ is a $d$-\textbf{Poincar\'e duality space} if there is a class $\theta \in H_d(X,k)$ such that, for all $\infty$-local systems $M$, the natural map $$-\cap \theta: H^\bullet(X,M) \to H_{d-\bullet}(X,M)$$ is a quasi-isomorphism.
\end{defn}
Note that the above definition is a chain-level version of classical Poincar\'e duality.

\begin{defn}
    Let $X$ be a path connected topological space. Say that $X$ is a $d$-\textbf{Frobenius space} if $C_\bullet X$ is a $d$-Frobenius coalgebra.
\end{defn}
By \ref{mainthm} combined with \ref{loopgroup}, $X$ is a Frobenius space if and only if $C_\bullet \Omega X$ is a Gorenstein algebra. Our main application is the following theorem, which shows that a sufficiently finite space is a Poincar\'e duality space if and only if it is a Frobenius space:

\begin{thm}\label{pdDict}
    Let $X$ be a path connected finitely dominated topological space. The following are equivalent:
    \begin{enumerate}
        \item $X$ is a $d$-Poincar\'e duality space.
        \item $X$ is a $d$-Frobenius space.
        \item $C_\bullet X$ is a $d$-Frobenius coalgebra.
        \item $C_\bullet X$ is a $d$-symmetric coalgebra.
        \item $C_\bullet \Omega X$ is a $d$-Gorenstein algebra.
        \item $C_\bullet \Omega X$ is a $d$-Calabi--Yau algebra.
    \end{enumerate}
\end{thm}

\begin{proof}
The equivalence of (2) and (3) is definitional, and we have already observed that (3) is equivalent to (5). The equivalence of (1) and (5) follows from \ref{gorduprop}, since by definition $X$ is a Poincar\'e duality space if and only if $C_\bullet\Omega X$ has Gorenstein duality. By \ref{hopfrmk} and \ref{moorermk} we may assume that $C_\bullet \Omega X$ is an involutive dg Hopf algebra, which is necessarily regular by \ref{fdsmprop}. So the equivalence of (5) and (6) follows from \ref{hopfthm}. The equivalence of (6) with (4) follows from \ref{mainthm} combined with \ref{loopgroup}.
\end{proof}

\begin{rmk}\label{pdnfdrmk}
Although it is not possible to formulate our notion of Poincar\'e duality space for non-finitely dominated topological spaces, we may still study the other conditions of \ref{pdDict}. Indeed, if $X$ is any path connected topological space then the conditions $(2)$, $(3)$ and $(5)$ are all equivalent. Moreover, $(4)$ is equivalent to $(6)$. 
\end{rmk}

\begin{rmk}
    Let $X$ be any topological space. Note that if $C_\bullet \Omega X$ is $d$-Gorenstein then we obtain vector space isomorphisms $H_i(X)^* \cong H^{d-i}(X)$, and so $X$ satisfies a weak form of Poincar\'e duality.
\end{rmk}

\begin{ex}
    Let $X\coloneqq \mathbb{C}\P^\infty$, so that $\Omega X$ is $S^1$. Then $C_\bullet \Omega X$ is quasi-isomorphic to $k[\epsilon]/\epsilon^2$, which is easily seen to be Gorenstein. Hence $X$ is a Frobenius space. However, $X$ is not a Poincar\'e duality space, since it is not finitely dominated.
\end{ex}

\begin{ex}[Rational homotopy theory]
       Take $k=\Q$. If $X$ is a simply connected space we put $W(X)\coloneqq \pi_*(\Omega X)\otimes \Q$, the Whitehead Lie algebra of $X$. If $U$ denotes the universal enveloping algebra then we have $U(WX) \simeq C_\bullet\Omega X$. If $L$ is a dg Lie algebra, say that $L$ is \textbf{Gorenstein} when $UL$ is. Then a path connected finitely dominated topological space $X$ is a Poincar\'e duality space precisely when $WX$ is a Gorenstein dg Lie algebra. Presumably this is (commutative-Lie) Koszul dual to a Frobenius condition on the minimal Sullivan model of $X$.
\end{ex}

\begin{ex}[Lie groups]
Let $G$ be a compact Lie group. Take $k=\Z/p$ and consider the dg Hopf algebra $C_\bullet G$. If $\pi_0G$ is a finite $p$-group, then $C_\bullet G$ is a Gorenstein algebra \cite[10.2]{DGI} and hence a CY algebra by \ref{hopfthm}. Moreover, since $C_\bullet G$ is proper, it is symmetric by \ref{cyimpsym} (this is well known when $G$ is finite, in which case we have $C_\bullet G \cong kG$). If $BG$ denotes the classifying space of $G$, then we have $BC_\bullet G \simeq C_\bullet BG$ since $\Omega BG \simeq G$ as topological groups. Hence the coalgebra $C_\bullet BG$ is both symmetric and Calabi--Yau. Similar statements can be derived for the $p$-completions $\hat G_p$ and $\hat{BG}_p$ using \cite[10.3]{DGI} and the fact that $\Omega \hat{BG}_p \simeq \hat G_p$ if $G$ is connected. It would be interesting to study similar properties for $p$-compact groups.
\end{ex}

\begin{ex}[String topology]
    
Let $X$ be a Poincar\'e duality space. Then \cite[5.3(3)]{BD} shows the well-known statement, originally due to Jones, that if $LX$ denotes the free loop space of $X$, then there is an $S^1$-equivariant quasi-isomorphism $$\HH_\bullet(C_\bullet \Omega X)\simeq C_\bullet(LX)$$ and Van den Bergh duality now gives a quasi-isomorphism $$\HH^\bullet(C_\bullet \Omega X)\simeq C_\bullet(LX)[-d].$$This recovers some results of \cite{vaintrob} for $X$ aspherical and \cite{ftstring} for $X$ simply connected.

\end{ex}

\begin{rmk}
    Instead of using path connected topological spaces in our above theorems, we could instead use pointed topological spaces and restrict to the path component of the point.
\end{rmk}

\begin{rmk}
    Let $X$ be a path connected space such that $H_\bullet(\Omega X)$ is proper; for example, $BG$ for a finite group $G$. Then the following are equivalent:
    \begin{enumerate}
        \item $X$ is a $d$-Frobenius space.
        \item $C_\bullet X$ is a $d$-Frobenius coalgebra.
        \item $C_\bullet X$ is a $d$-symmetric coalgebra.
        \item $C_\bullet \Omega(X)$ is a $d$-Gorenstein algebra.
        \item $C_\bullet \Omega(X)$ is a $d$-Calabi--Yau algebra.
    \end{enumerate}
    To see this, follow the proof of \ref{pdDict} and use \ref{hopfthm}(2) instead of \ref{hopfthm}(1). We suggest that any of the above equivalent conditions are a replacement for the concept of Poincar\'e duality space for spaces that are not finitely dominated but do have finite loop homology. In this sense, spaces with finite loop homology are `Koszul dual' to finitely dominated spaces.
\end{rmk}

\subsection{Simply connected spaces}
We refine our Poincar\'e duality results when the spaces under consideration are simply connected. Note that a simply connected finitely dominated space is homotopy equivalent to a finite CW complex, since as in \ref{wallfiniteness} its Wall finiteness obstruction necessarily vanishes. Recall from \ref{sccd} the definition of a simply connected coalgebra.

\begin{lem}\label{ahmodel}
    Let $X$ be a simply connected finitely dominated space. Then $C_\bullet X$ is weakly equivalent to a simply connected coalgebra.
\end{lem}
\begin{proof}
    Up to homotopy equivalence we may assume that $X$ is a finite CW complex with one $0$-cell and no $1$-cells. Adams and Hilton \cite{adamshilton} proved that $C_\bullet\Omega X$ is quasi-isomorphic to an algebra of the form $A=(T(V),d)$ where $V$ is a graded vector space concentrated in strictly negative cohomological degrees. Hence $BA\simeq BC_\bullet\Omega X \simeq C_\bullet X$ is a simply connected coalgebra for degree reasons.
\end{proof}

\begin{thm}\label{fhtthm}
    Let $X$ be a simply connected finitely dominated space. Then the following are equivalent:
    {\begin{enumerate}
        \item $X$ is a $d$-Poincar\'e duality space.
        \item $C^\bullet X$ is a $(-d)$-Frobenius algebra.
        \item $C^\bullet X$ is a $(-d)$-symmetric algebra.
        \item $C^\bullet X$ is a $(-d)$-Gorenstein algebra.
        \item $C^\bullet X$ is a $(-d)$-Calabi--Yau algebra.
    \end{enumerate}}
\end{thm}
\begin{proof}
     Since $X$ is finitely dominated, $C_\bullet\Omega X$ is smooth, and hence $C_\bullet X$ is a strongly proper coalgebra by \ref{kdpropc}. Moreover, since $X$ was simply connected, we may assume that $C_\bullet X$ is simply connected by \ref{ahmodel}. Recalling that  $X$ is a $d$-Poincar\'e duality space if and only if $C_\bullet X$ is a $d$-Frobenius (equivalently symmetric) coalgebra (\ref{pdDict}), the equivalence of the first three statements follows from \ref{propLD}, since $C^\bullet X$ is the linear dual of $C_\bullet X$. The implication $(2)\implies(4)$ is \ref{frobisgor}. Since $X$ is homotopy equivalent to a finite CW complex, cellular cohomology gives us a finite dimensional algebra $A$ and a quasi-isomorphism $A\simeq C^\bullet X$. The implication $(4) \implies (2)$ now follows from \ref{gorimpfrob}. Finally, the equivalence of $(3)$ and $(5)$ is \ref{cyimpsym}.
\end{proof}

\begin{ex}\label{scex}
The simply connectedness assumption in the above theorem cannot be dropped. Let $X$ be an acyclic space with nontrivial fundamental group (such spaces can be obtained as the fibre of the Quillen plus construction). Then $C^\bullet X \simeq k$ is a $0$-Frobenius algebra. If $X$ was a $0$-Poincar\'e duality space, then for all $\infty$-local systems $\mathcal{F}$ on $X$ we would have isomorphisms $H^i(X,\mathcal{F})\cong H_{-i}(X,\mathcal{F})$. We will show that this leads to a contradiction. Indeed, let $\rho: k\pi_1(X) \to V$ be a nontrivial representation and let $\mathcal{V}$ be the corresponding local system. Since $H_{-1}(X,\mathcal{V})\cong 0$, we need only show that  $H^1(X,\mathcal{V})$ need not vanish. But the natural t-structure on $D(C_\bullet \Omega X)$ with heart the category of $k\pi_1$-representations yields a natural isomorphism $H^1(X,\mathcal{V})\cong H^1(k\pi_1,V)$ with group cohomology, which need not vanish.
\end{ex}
The above gives a new proof of a result of F\'elix, Halperin, and Thomas on Gorenstein spaces; recall that a topological space is \textbf{Gorenstein} if $C^\bullet X$ is a Gorenstein algebra.
\begin{prop}
{(cf.\ \cite{FHT})}\label{gspthm}
    Let $X$ be a path connected finitely dominated topological space. If $X$ is a Poincar\'e duality space, then $X$ is a Gorenstein space. The converse holds if $X$ is simply connected.
\end{prop}
\begin{proof}
As in the proof of \ref{fhtthm}, $C_\bullet X$ is a proper coalgebra, and hence if $X$ is a Poincar\'e duality space then $C^\bullet X$ is a Frobenius algebra by \ref{propLD}. Moreover, a Frobenius algebra is Gorenstein by \ref{frobisgor}, which proves the first claim. The second claim follows from \ref{fhtthm}.
\end{proof}

\begin{rmk}\label{gsprmk}
Let $X$ be a path connected finitely dominated topological space. Then for the statements of \ref{fhtthm}, statement (1) implies all the others. Statement (2) implies statement (4), and the converse is true if $C^\bullet X$ admits a finite dimensional model. Statements (3) and (5) are equivalent.
\end{rmk}

	\begin{footnotesize}
	\bibliographystyle{alpha}
	\bibliography{references.bib}
\end{footnotesize}

	\end{document}